\documentclass[oneside]{amsart}
\usepackage{geometry}

\usepackage{graphicx}
\usepackage[utf8]{inputenc}
\usepackage{nicefrac}
\usepackage{mathabx}
\usepackage{subcaption}
\usepackage{amsrefs}
\usepackage{esint}
\usepackage{mathtools}
\usepackage{graphicx}

\usepackage{wrapfig}
\usepackage{xcolor}
\usepackage{enumitem}
\usepackage{verbatim}
\usepackage{hyperref}
\usepackage{cleveref}
\usepackage{subcaption}
\usepackage{autonum}
\usepackage[normalem]{ulem}

\newtheorem{theorem}{Theorem}[section]
\newtheorem{lemma}[theorem]{Lemma}

\theoremstyle{definition}

\newtheorem{example}[theorem]{Example}

\newtheorem{cor}[theorem]{Corollary}

\newtheorem{question}[theorem]{Question}

\allowdisplaybreaks
\theoremstyle{remark}
\newtheorem{remark}[theorem]{Remark}

\numberwithin{equation}{section}

\allowdisplaybreaks

\newcommand{\abs}[1]{\lvert#1\rvert}

\newcommand{\dist}{\mathrm{dist}}
\newcommand{\diam}{\operatorname{diam}}

\makeatletter
\newcommand*\bigcdot{\mathpalette\bigcdot@{.5}}
\newcommand*\bigcdot@[2]{\mathbin{\vcenter{\hbox{\scalebox{#2}{$\m@th#1\bullet$}}}}}
\makeatother

\usepackage{amsmath}
\usepackage{braket} 
\newcommand{\R}{\mathbb{R}}
\newcommand{\N}{\mathbb{N}}
\newcommand{\Z}{\mathbb{Z}}

\def\red{\textcolor{red}}

\makeatletter
\@namedef{subjclassname@2020}{%
  \textup{2020} Mathematics Subject Classification}
\makeatother

\begin{document}
\title{Short closed geodesics and the Willmore energy}


\author[M.~Müller]{Marius Müller}
\address[M.~Müller]{Institute of Mathematics, University of Augsburg,  Universitätsstraße 14, 86159 Augsburg.}
\email{marius1.mueller@uni-a.de}

\author[F.~Rupp]{Fabian Rupp}
\address[F.~Rupp]{Faculty of Mathematics, University of Vienna, Oskar-Morgenstern-Platz 1, 1090 Vienna, Austria.}
\email{fabian.rupp@univie.ac.at}
\author[C.~Scharrer]{Christian Scharrer}
\address[C.~Scharrer]{Institute for Applied Mathematics, University of Bonn, Endenicher Allee 60, 53115 Bonn, Germany.}
\email{scharrer@iam.uni-bonn.de}




\date{\today}


\begin{abstract}
We prove a lower bound on the length of closed geodesics 
for spheres with Willmore energy below $6\pi$. The energy threshold is optimal and the inequality cannot be extended to surfaces of higher genus. Moreover, we discuss consequences for the injectivity radius. 
\end{abstract}
\maketitle

\section{Introduction}
Let $(\Sigma,g)$ be a 
Riemannian surface.
The length of an immersed curve $\gamma\colon (a,b) \rightarrow \Sigma$ is given by
\begin{equation}
    \mathcal L_g(\gamma)\vcentcolon=\int_\gamma \mathrm{d} s_g\vcentcolon= \int_a^b\sqrt{g(\dot\gamma,\dot\gamma)}\,\mathrm dt>0.
\end{equation}
Immersed curves satisfying $\nabla_{\dot\gamma}\dot\gamma=0$ are critical points of the length functional $\mathcal{L}_g$. Referred to as \emph{geodesics}, they 
play a key role in the local and
global analysis of Riemannian manifolds.
A \emph{closed} geodesic is a periodic geodesic $\gamma\in C^\infty(\mathbb{S}^1;\Sigma)$.
Minimizing $\mathcal{L}_g$ in a nontrivial homotopy class yields the existence of closed geodesics on surfaces with nontrivial topology. For spheres, this was shown by Birkhoff  \cite{Birkhoff1917TAMS}.
We may thus define the \emph{length of the shortest closed geodesic}
\begin{equation}\label{eq:def_ell}
    \ell(\Sigma,g)\vcentcolon = \inf\{\mathcal L_g(\gamma)\mid \gamma\colon\mathbb S^1\to(\Sigma,g)\text{ closed geodesic}\}.
\end{equation}
An important task is to relate $\ell(\Sigma,g)$ to other geometric quantities of the surface, e.g.\ its area, diameter, or curvature. Indeed, 
the 2-dimensional version of a question asked by Gromov \cite{Gromov1983} in 1983 
in all dimensions is an upper bound of the form
\begin{align}\label{eq:ell_upper_bound}
    \ell(\Sigma,g)\leq C \sqrt{\mu_g(\Sigma)}.
\end{align}
 Here $\mu_g$ denotes the Riemannian measure induced by the metric $g$.
For $\Sigma=\mathbb{S}^2$, Croke \cite{Croke88} proved \eqref{eq:ell_upper_bound} with a nonoptimal universal constant. 
For further literature, see \cites{MR2231630,CalabiCao92,CrokeKatzSurvey}, and references therein.

On the other hand, lower bounds on $\ell(\Sigma,g)$ play a vital role, for example in Cheeger's finiteness theorem \cite{Cheeger1970}, where an $L^\infty$-bound on the curvature is assumed.
If the Gauss curvature is \emph{pinched}, i.e.\ if
\begin{equation}\label{eq:intro:pinched}
    1/4 \leq K \leq 1 \quad \text{on }\Sigma,
\end{equation}
a result by Klingenberg \cite{Klingenberg61} yields that the injectivity radius $i(\Sigma,g)$ satisfies $i(\Sigma,g)\geq \pi,$ and thus one immediately concludes $\ell(\Sigma,g)\geq 2\,i(\Sigma,g)\geq 2\pi$.  

In this paper, we consider the case where $(\Sigma,g)$ is isometrically immersed in $\R^n$. We show that in order to bound $\ell(\Sigma,g)$ from below, the pointwise bounds on the Gauss curvature, cf.\ \eqref{eq:intro:pinched}, can be replaced by an upper bound on the $L^2$-norm of the mean curvature. We emphasize that for variational settings such $L^2$-curvature bounds are much more favorable than pointwise bounds. 

To state our main result, we briefly recall the notion of curvature for immersed surfaces in $\R^n$. For an immersion $f\colon\Sigma \to\R^n$ of a 2-dimensional manifold $\Sigma$, we denote by $g=g_f$ the pullback metric along $f$ of the Euclidean inner product $\langle\cdot,\cdot\rangle$ in $\R^n$. Moreover, the second fundamental form $A=A_f$ gives rise to the mean curvature vector 
 $H=H_f=\mathrm{tr}_{g_f}(A_f)$ and the Gauss curvature $K=K_f$, defined by $K \vcentcolon = \langle A(\tau_1,\tau_1), A(\tau_2,\tau_2)\rangle - \langle A(\tau_1,\tau_2), A(\tau_1,\tau_2)\rangle$, where $\tau_1,\tau_2$ is an orthonormal frame.
 Note that $K_f=K_{g_f}$ is determined by the metric by Gauss's Theorema Egregium. Whenever there is no ambiguity, we will omit the dependence on $f$ and $g_f$.
The area and Willmore energy of $f$ are defined by
\begin{equation}
    \mathcal A(f) \vcentcolon = \mu(\Sigma),\qquad \mathcal W(f)\vcentcolon= \frac{1}{4}\int_\Sigma|H|^2\,\mathrm d\mu.
\end{equation}
If $\Sigma$ is closed, then $\mathcal W(f) \geq 4\pi$ with equality if and only if $f$ parametrizes a round sphere, see 
\cite[Theorem 7.2.2]{Willmore_Riemannian} and \cite[Theorem 3]{Chen1971}.
Thus, $\mathcal{W}$ quantifies the defect of a surface to be round.

\subsection{Main result}
Our main result provides a lower bound on $\ell(\Sigma,g)$ where  
the constant depends only on the Willmore energy, without assuming any pointwise bounds on the curvature. Since the Willmore energy is critical for the Sobolev embedding, hence does not control the metric uniformly, the existence of such a lower bound is 
nontrivial.
\begin{theorem}\label{thm:main}
    There exists a constant $C(n)>0$ such that for all immersions $f\colon \mathbb{S}^2\to \R^n$ with $\mathcal{W}(f)< 6\pi$
    we have
    \begin{align}
        \ell(\mathbb{S}^2,g_f) \geq C(n) \left(6\pi - \mathcal{W}(f)\right) \sqrt{\mathcal{A}(f)}.
    \end{align}
\end{theorem}
This describes a level of roundness of surfaces with small Willmore energy which resembles the De Lellis--Müller rigidity result \cite{DeLellisMueller05} and its higher codimension analogue by Lamm--Schätzle \cite{LammSchaetzle14}. Note that for $f\colon\mathbb{S}^2\to\R^n$, $n\geq 5$, the condition $\mathcal{W}(f)< 6\pi$ is equivalent to the assumption used in \cite[Theorem 1.2]{LammSchaetzle14} to control the conformal parametrization.
On the other hand, if $\Sigma$ has genus $p\geq 1$, and $f\colon \Sigma \to\R^3$ is an immersion, then we have $\mathcal{W}(f)\geq 2\pi^2 > 6\pi$ by the resolution of the Willmore conjecture due to Marques--Neves \cite{MarquesNeves14}.

We do not know the optimal constant in \Cref{thm:main}. However, our proof yields a lower bound for $C(n)$, see \Cref{rem:opt_C}, which is bounded away from zero, uniformly in $n\in\N$. See also \Cref{thm:main_inj} if the shortest closed geodesic is \emph{embedded}, i.e.\  has no self-intersections.
Nevertheless, the energy threshold of $6\pi$ in \Cref{thm:main} is sharp, see \Cref{ex:6pi_opt}. Moreover, as a consequence of the noncompactness of the invariance group of the Willmore energy, the inequality cannot be generalized to higher genus surfaces in any codimension, see \Cref{ex:genus_zero}, so also $\Sigma=\mathbb{S}^2$ is sharp.

The idea of the proof of \Cref{thm:main} is easily illustrated if only \emph{embedded} closed geodesics are considered. 
Indeed, an embedded closed geodesic splits the sphere $\mathbb{S}^2$ into two topological disks both of which have Willmore energy at least $2\pi$ as a consequence of the Gauss--Bonnet theorem 
and the fact that $\frac{1}{4}|H|^2\geq K$. If the geodesic becomes very short, then any disk with a uniform lower area bound contributes nearly $4\pi$ Willmore energy as a consequence of elementary diameter bounds and Simon's monotonicity formula. For the details, 
see the proof of \Cref{thm:main_inj}.
The main difficulty arises from the case where the shortest closed geodesic has self-intersections, which it might have, in general, cf.\  \cite[p.\ 31]{Pitts}. However, also in this case we can identify a suitable tiling of $\mathbb{S}^2$ and control the curvature in the resulting parts individually, see \Cref{sec:tiling}.

Combining \Cref{thm:main} with an estimate due to Klingenberg (see for instance \cite[Lemma 6.4.7]{Petersen}), we obtain the following lower bound on the injectivity radius. 

\begin{cor}\label{cor:main}
    With $C(n)>0$ as in \Cref{thm:main}, for all immersions $f\colon\mathbb{S}^2\to\R^n$ 
    we have
    \begin{equation}\label{eq:Kling_inj_est}
        i(\mathbb S^2,g_f) \geq \min\left\{\frac{\pi}{\sqrt{\max K_f}}, \frac{C(n)}{2}(6\pi - \mathcal W(f))\sqrt{\mathcal A(f)}\right\}.
    \end{equation}
\end{cor}

\subsection{Optimality discussion}\label{sec:optimality}

We now illustrate the optimality of the assumptions in \Cref{thm:main} by a set of examples. 

\begin{example}[Optimality of the $6\pi$-threshold]\label{ex:6pi_opt}
    For each $N\in\N$ there exists a family of smooth embeddings $f_a\colon\mathbb S^2\to\R^3$, $0<a<a_0$, and corresponding closed geodesics $\gamma_a\colon\mathbb S^1\to(\mathbb S^2,g_{f_a})$, each of which having exactly $N$ distinct self-intersections, such that 
    \begin{equation}\label{eq:intro:ex:sharpness}
        \lim_{a\to0}\mathcal L_{g_{f_a}}(\gamma_a)/\sqrt{\mathcal A(f_a)} = 0,\qquad \lim_{a\to 0}\mathcal W(f_a) = 6\pi.
    \end{equation}
\end{example}

We sketch the construction. Consider the following  pieces of surfaces.
A capped unit sphere; a piece of a catenoid; a cylinder of radius $a>0$ and height $2a$; half a sphere of radius $a$:
\begin{align}
    \Sigma_1 &= \{(x_1,x_2,x_3)\in\R^3\mid x_1^2+x_2^2+x_3^2=1,\,x_3 \leq 1-s_a \};\\
    \Sigma_2 &= \{(a\cosh(t/a)\cos(\theta),a\cosh(t/a)\sin(\theta),t)\mid t\in [-t_a,0], \theta\in [0,2\pi)\};\\
    \Sigma_3&=\{(x_1,x_2,x_3)\in\R^3\mid x_1^2+x_2^2 = a^2,\,|x_3|\leq a\};\\
    \Sigma_4&=\{(x_1,x_2,x_3)\in\R^3\mid x_1^2+x_2^2+x_3^2=a^2, x_3 \geq 0\}.
\end{align}
Clearly, $\Sigma_3$ contains a closed circular geodesic $\gamma_a$ of length $2\pi a$. If $s_a,t_a$ are chosen in a suitable way, after rotation and translation the four pieces can be glued together with $C^{1,1}$-regularity.
After smoothing at the gluing regions without affecting the geodesic $\gamma_a$, 
we thus obtain immersions $f_a\colon \mathbb{S}^2\to\R^3$ with corresponding closed geodesics $\gamma_a\colon \mathbb{S}^1\to (\mathbb{S}^2,g_{f_a})$. Noting that the catenoidal part $\Sigma_2$ carries zero Willmore energy, we conclude that
 \begin{equation}
       \lim_{a\to 0}\mathcal L_{g_{f_a}}(\gamma_a)/\sqrt{\mathcal{A}(f_a)} = 0,\qquad \lim_{a\to 0}\mathcal W(f_a) = 6\pi,
\end{equation}
\begin{figure}
     \centering
     \hfill
     \begin{subfigure}[b]{0.45\textwidth}
         \centering
         \def\svgwidth{0.8\linewidth}
         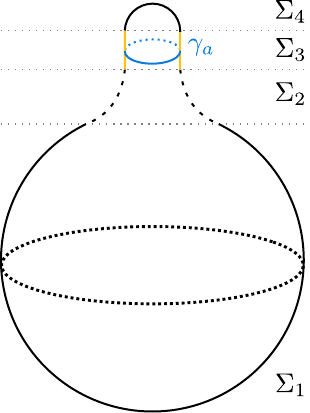
         \caption{A closed geodesic $\gamma_a$ of length $2\pi a$ on $\Sigma_3$.}
          \label{fig:Sigma_a}
     \end{subfigure}
     \hfill
     \begin{subfigure}[b]{0.45\textwidth}
         \centering
         \def\svgwidth{0.8\linewidth}
         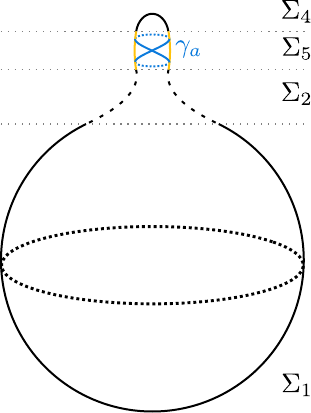
         \caption{A 
         geodesic $\gamma_a$ on $\Sigma_5$ with self-intersection.}
         \label{fig:Sigma_5}
     \end{subfigure}
     \hfill
     \caption{Short closed geodesics on $\Sigma_3$ and $\Sigma_5$.}
     \label{fig:Counterexp}
 \end{figure}%
see \Cref{fig:Sigma_a}.
For $b>0$, we may replace $\Sigma_3$ with a suitably small spheroid 
(i.e.\ a rotationally symmetric ellipsoid) of the form
\begin{equation}\label{eq:intro:spheroidal}
    \Sigma_5=\{a(\cos(t)\cos(\theta),\cos(t)\sin(\theta), b\sin(t))\mid |t|\leq a,\, \theta\in [0,2\pi)\}, 
\end{equation}
where $\Sigma_2$, $\Sigma_4$ need to be adjusted accordingly, see \Cref{fig:Sigma_5}.
Choosing $b=b(a,N)$ as in \Cref{lem:spheroids} below (with $\varepsilon=a$),
we may even achieve that the geodesics $\gamma_a$ have exactly $N\in\N$ distinct self-intersections, cf.\ \Cref{fig:spheroid_plots}, while still being arbitrarily short by \eqref{eq:spheroid_length_geod}. This completes the discussion of \Cref{ex:6pi_opt} up to  \Cref{lem:spheroids} which we prove in \Cref{sec:spheroids} below.

\Cref{thm:main} cannot be generalized to surfaces with higher genus. Indeed, by \cite{Simon,Kusner,BauerKuwert} each orientable closed surface of genus $p\geq 1$ admits an embedding 
into $\R^n$, $n\geq 3$, of minimal Willmore energy. After replacing some part of a minimizer with a small flat disk, a suitable Möbius transformation makes the minimizer look like a unit sphere with $p$ tiny handles attached. If done properly, this even leads to the existence of arbitrarily short geodesics that are null-homotopic. The details of the following example are discussed in \Cref{sec:optimality_genus} below.

\begin{example}[Optimality of zero genus]\label{ex:genus_zero}
    For all $\varepsilon>0$ and each closed, connected, and orientable surface $\Sigma$ of genus $p\geq 1$ there exists an embedding $f_\varepsilon\colon\Sigma\to\mathbb \R^n$ such that 
    \begin{equation}
        \mathcal W(f_\varepsilon) \leq \min 
        \{\mathcal W(f)\mid f\colon\Sigma\to\R^n\text{ immersion}\} +\varepsilon,
    \end{equation}
    and a 
    null-homotopic geodesic $\gamma\colon \mathbb S^1 \to \mathbb (\Sigma,g_{f_\varepsilon})$ 
    such that we have
    $\mathcal L_{g_{f_\varepsilon}}(\gamma)/\sqrt{\mathcal{A}(f_\varepsilon)}<\varepsilon$.
\end{example}

On round spheres, the absolute minimizer of $\mathcal{W}$, 
all closed geodesics are embedded.
However, small Willmore energy does not rule out the existence of closed geodesics with self-intersections.
\begin{example}[Small Willmore energy and self-intersections]
\label{ex:13}
    For all $N\in\N$ and $\varepsilon>0$ there exists an embedding $f\colon\mathbb S^2\to\mathbb R^3$ 
    which contains a closed geodesic with exactly $N$ distinct self-intersections and $\mathcal W(f)<4\pi + \varepsilon$.
\end{example}

Indeed, consider a thin slab of width $2\varepsilon$ of the spheroid $\Sigma_5$ as in \eqref{eq:intro:spheroidal} with $a=1$, centered around the equator. Choosing $b=b(N,\varepsilon)$ as in \Cref{lem:spheroids},  yields the existence of closed geodesics on this slab with exactly $N\in \N$ self-intersections. The Willmore energy of the slab is of order $\varepsilon$. Hence gluing it together with two spherical caps, we see that the Willmore energy is arbitrarily close to $4\pi$, and the statement follows from \Cref{lem:spheroids}.

Since these geodesics intersect many times, they 
are unlikely to realize \eqref{eq:def_ell}. While shortest closed geodesics might 
have self-intersections in general (recall the \emph{tree-legged starfish}, see \cite[p.\ 31]{Pitts}, \cite[Figure 2]{CalabiCao92}), 
they are embedded if $K\geq 0$ by a result of Calabi--Cao \cite{CalabiCao92}; however, such a pointwise curvature control is impossible to deduce from smallness of the Willmore energy, see \Cref{ex:Toro} below. This naturally leads to the following open problem.
\begin{question}\label{question:16}
    Let $f\colon\mathbb S^2 \to \R^n$ be an immersion 
    such that the shortest closed geodesic on $(\mathbb S^2,g_f)$ has self-intersections. Does there exist a universal constant $C>4\pi$ such that $\mathcal{W}(f)\geq C$?
\end{question}
While we do not know exactly the Willmore energy of the three-legged starfish, heuristically, each of its legs should have at least the Willmore energy of a hemisphere. Thus, it might be tempting to conjecture that $C\geq 6\pi$ in \Cref{question:16}. 

Finally, we discuss the optimality of \Cref{cor:main}. The first term in the minimum in \eqref{eq:Kling_inj_est} could be dropped if the Gauss curvature was suitably bounded from above. However, this is generically not possible, even for arbitrarily small Willmore energy. 

\begin{example}[Unbounded Gauss curvature]\label{ex:Toro}
    For all $\varepsilon>0$ there exists an embedding $f\colon \mathbb S^2\to\mathbb R^3$ with $\mathcal W(f)<4\pi + \varepsilon$ and $\mathcal A(f) = 4\pi$ such that $\max K_f>1/\varepsilon$ and $\min K_f<-1/\varepsilon$.
\end{example}
This follows from flattening a small disk on a unit sphere as in \Cref{lem:replacement} and then replacing it with a part of the graph of the function $u$ from Toro's example \cite[Example 1]{Toro}. 
Upon rescaling $u$ and smoothing in a small neighborhood of the origin, we obtain an immersion with arbitrarily small Willmore energy, see \cite[Example 1]{DGR2017}. The statement of \Cref{ex:Toro} thus follows from the fact that the Gauss curvature becomes very large near the origin as we show in \Cref{lem:Toro}. 

The right hand side of \eqref{eq:Kling_inj_est} for \Cref{ex:Toro} equals $\pi/\sqrt{\max K_f}$ which is less than $\pi\sqrt{\varepsilon}$. However, we could not verify that \Cref{ex:Toro} indeed has small injectivity radius which leaves the following problem unsolved.

\begin{question}\label{question:injectivity_radius}
    Let $f_k\colon\mathbb S^2\to\mathbb R^n$ be a sequence of immersions such that
    $i(\mathbb S^2,g_{f_k}) / \sqrt{\mathcal{A}(f_k)}\to 0$ as $k\to\infty$.
    Does there exist a universal constant $C>4\pi$ such that $\liminf_{k\to\infty} \mathcal W(f_k) \geq C$?
\end{question}
Note that if $\limsup_{k\to\infty}\mathcal{W}(f_k)<6\pi$, by \Cref{cor:main} the Gauss curvatures of such a sequence necessarily degenerate in the sense that $\lim_{k\to\infty}\max K_{f_k} = \infty$.

\section{Preliminaries}
Throughout this article, we denote by $\Sigma$ a compact, connected, and orientable two-dimensional smooth manifold without boundary.
Let $f\colon\Sigma\to \R^n$ be an immersion.
For an open set $D \subset \Sigma$ we define
\begin{equation}
    \mathcal{A}(f,D) \vcentcolon = \int_D 1 \; \mathrm{d}\mu , \quad 
    \mathcal{W}(f,D) \vcentcolon = \frac{1}{4} \int_D  |H|^2 \; \mathrm{d}\mu.
\end{equation}
Further, we denote $PS(\mathbb{S}^1;\Sigma) \vcentcolon= \{ \gamma \in C^0(\mathbb{S}^1;\Sigma) : \gamma \textrm{ piecewise smooth} \}$. We view $\mathbb{S}^1 = \R/ 2\pi\Z$ and say that $\gamma\colon \mathbb{S}^1\to\Sigma$ is \emph{simply covered} if the period of $\gamma\colon [0,2\pi]\to\Sigma$ is $2\pi$. In the sequel we will often consider domains $D$ that satisfy 
\begin{align}\label{eq:adm_boundary} \tag{\texttt{A}}
\begin{split}
  & \textrm{Each connected component of }  \partial D  \; \textrm{can be parametrized by}\\
  & \textrm{a piecewise immersed simply covered curve $\gamma \in PS(\mathbb{S}^1;\Sigma)$}\\
  &\textrm{having only finitely many self-intersections}. 
\end{split}
\end{align}
An example for such a domain $D$ is depicted in \Cref{fig:D_nonemb_bdry}.
\emph{Piecewise immersed} in 
\eqref{eq:adm_boundary} means that $\dot{\gamma}(t) \neq 0$ for all but finitely many $t \in \mathbb{S}^1$. 
If $D$ satisfies \eqref{eq:adm_boundary} then it is easily seen to be a \emph{manifold with thin singular set}, in the sense of \cite[Chapter XII, 3] {AmannEscher}. The length $\mathcal{L}_g(\partial D)$ is defined as the sum of the lengths of the parametrizations chosen as in \eqref{eq:adm_boundary}. 
If it is possible to choose all curves $\gamma$ in \eqref{eq:adm_boundary} without self-intersections then we say that $D$ has (piecewise smooth) \emph{embedded}  \emph{boundary}.  
Moreover, since $f\colon (\Sigma, g_f)\to (\R^n, \langle\cdot,\cdot\rangle)$ is an isometry, the intrinsic and extrinsic lengths coincide
\begin{equation}\label{eq:star}
    \mathcal{L}(f\circ \gamma) \vcentcolon= \mathcal{L}_{\langle\cdot,\cdot\rangle}(f\circ \gamma) = \mathcal{L}_{g_f}(\gamma)
\end{equation}
 for any immersed curve $\gamma\colon(a,b)\to \Sigma$.
 The (extrinsic) diameter of $A\subset \R^n$ will be denoted with
 \begin{equation}
     \diam(A)\vcentcolon=\sup_{x,y \in A}|x-y|.
 \end{equation}
\subsection{The Gauss divergence formula and consequences}
Suppose that $f\colon \Sigma \rightarrow \mathbb{R}^n$ is an immersion
 and suppose that $D \subset \Sigma$ 
satisfies \eqref{eq:adm_boundary}. 
For $\phi \in C^\infty(\Sigma;\mathbb{R}^n)$ we define $\mathrm{div}\phi(x) \vcentcolon = \sum_{i = 1}^2 \langle \partial_{\tau_i} \phi(x), \partial_{\tau_i} f(x)\rangle$, where $\{ \tau_1,\tau_2\}$ is an orthonormal basis of $T_x \Sigma$. We define the \emph{outward pointing unit normal} $\nu(x)= \nu_D(x)$ as in \cite[p. 412]{AmannEscher} for all but finitely many points in $x \in \partial D$. For such $x$ we also consider the \emph{unit conormal} $\eta(x) \vcentcolon = \partial_{\nu} f(x)\in \mathbb{S}^{n-1}$. With the aid of \cite[Theorem XII.3.15]{AmannEscher} one can derive the \emph {divergence theorem} for immersed surfaces 
\begin{equation}
    \int_D \mathrm{div} \phi \; \mathrm{d}\mu = \int_{\partial D} \langle\phi, \eta\rangle \;  \mathrm{d}s- \int_D \langle\phi, H\rangle \; \mathrm{d}\mu \qquad \forall  
    \phi \in C^\infty(\Sigma; \mathbb{R}^n). \label{eq:divsatz}
\end{equation}
Here $\mathrm{d}s = \mathrm{d}s_{g_f}$ denotes the arc length element of the parametrizations $\gamma: \mathbb{S}^1 \rightarrow (\Sigma,g_f)$ chosen as in \eqref{eq:adm_boundary}.
The divergence theorem \eqref{eq:divsatz} implies \emph{Simon's monotonicity formula}, which has first been obtained in \cite{Simon} and later been extended to the case of manifolds with boundary, cf. \cite[Lemma A.3]{Tristan}, \cite{Volkmann}. 
It allows for useful diameter estimates for subsets satisfying \eqref{eq:adm_boundary}.
\begin{lemma}\label{lem:22_new}
Let $f\colon \Sigma \rightarrow \mathbb{R}^n$ be an immersion and suppose that $D \subset \Sigma$
satisfies \eqref{eq:adm_boundary}.
\begin{enumerate}[label=(\roman*)]
    \item\label{item:22_new_1} We have
    \begin{align}
         \diam(f(D)) \geq \frac{2\mathcal{A}(f,D)}{\mathcal{L}_{g}(\partial D) + 2\mathcal{W}(f,D)^\frac{1}{2} \mathcal{A}(f,D)^\frac{1}{2}}.
    \end{align}
\item\label{item:22_new_2} If $\partial D$ is connected, then there exists $x_0\in D$ such that 
\begin{align}
    \mathrm{dist}(f(x_0),f(\partial D)) \geq \frac{1}{2} \left( \diam(f(D))- \frac{\mathcal{L}_{g}(\partial D)}{2}\right).
\end{align}
\item\label{item:22_new_3} For each $x_0 \in D$ we have
\begin{equation}
    4\pi \leq \mathcal{W}(f,D) + 2 \int_{\partial D} \frac{\mathrm{d}s(x)}{|f(x)-f(x_0)|}.
\end{equation}
    \item \label{item:22_new_4} If $n\ge4$ then
\begin{equation}
    \sqrt{\mathcal A(f,D)} \le C_n\Bigl(\int_D|H|\,\mathrm d\mu + 2\mathcal L_g(\partial D)\Bigr)
\end{equation}
for $C_n\vcentcolon=\min\{\sqrt{\frac{n-2}{32\pi}},\frac{75}{\sqrt{\pi}}\}$.
\end{enumerate}
\end{lemma}
\begin{proof}
    For \ref{item:22_new_1}, let $z \in \partial D$ be arbitrary. We use $\phi(x) \vcentcolon = f(x) - f(z)$ in \eqref{eq:divsatz} and find (with $\mathrm{div}\phi \equiv 2$) 
\begin{equation}
   2 \mathcal{A}(f,D) = \int_{\partial D} \langle f(x)-f(z), \eta(x) \rangle \; \mathrm{d}s(x) - \int_D \langle f(x)-f(z), {H}(x) \rangle \; \mathrm{d} \mu(x). 
\end{equation}
Estimating $|f(x)-f(z)|\leq \diam(f(D))$ we find
\begin{equation}
    2 \mathcal{A}(f,D) \leq \diam(f(D)) \mathcal{L}_{g}(\partial D) +  \diam(f(D)) \int_D |H| \; \mathrm{d}\mu . 
\end{equation}
Using that $\int_D |H| \; \mathrm{d} \mu \leq  2\mathcal{W}(f,D)^\frac{1}{2} \mathcal{A}(f,D)^\frac{1}{2}$, the claim follows. \newline\indent
For \ref{item:22_new_2}, let $x,y \in D$, $z_1,z_2\in \partial D$. We have $$
    |f(x)-f(y)| \leq |f(x)-f(z_1)| + \diam (f(\partial D)) + |f(z_2)-f(y)|.$$
Assuming $\dist(f(x),f(\partial D)) \geq \dist (f(y),f(\partial D))$ this implies
\begin{align}
    |f(x)-f(y)| \leq 2~\dist(f(x), f(\partial D)) + \diam (f(\partial D)).
\end{align}
Let $z_1,z_2\in \partial D$ with $|f(z_1)-f(z_2)| = \diam(f(\partial D))$. 
Since $D$ satisfies \eqref{eq:adm_boundary} and $\partial D$ is connected, there exists a single curve $\gamma \in PS(\mathbb{S}^1;\Sigma)$ as in \eqref{eq:adm_boundary} such that $\partial D = \gamma(\mathbb{S}^1)$. At least one of the arcs of $\gamma$ connecting $z_1$ to $z_2$ has length at most $\mathcal{L}_g(\partial D)/2$. Using \eqref{eq:star}, we thus find
\begin{align}
    |f(x)-f(y)| &\leq 2~\dist(f(x), f(\partial D)) + |f(z_1)-f(z_2)|\\
    &\leq 2~\dist(f(x), f(\partial D)) + \frac{\mathcal{L}_g(\partial D)}{2}.
\end{align}
It follows that $\diam(f(D)) \leq 2 \max_{x\in D} \dist(f(x),f(\partial D)) +\frac{\mathcal{L}_g(\partial D)}{2}$. Since the maximum on the right hand side is attained in $D$, the claim follows.
\newline\indent
Item \ref{item:22_new_3} is an immediate consequence of the monontonicity formula with boundary, see \cite[Equation (18)]{Volkmann}. 
\newline\indent
Item \ref{item:22_new_4} is originally due to Michael--Simon \cite[Theorem 2.1]{MichaelSimon}. A general version for surfaces with boundary which is immediately applicable for Setting \eqref{eq:adm_boundary} is given in \cite[Theorem 3.5]{MenneScharrer18} with the constant $C_n = 75/\sqrt{\pi}$. A version with the potentially smaller constant $C_n = \sqrt{(n-2)/32\pi}$ depending on the codimension is given in \cite[Theorem 1]{Brendle21JAMS} for surfaces whose  boundary is $C^{2,\gamma}$-regular for some $0<\gamma<1$. This theorem can also be applied in Setting \eqref{eq:adm_boundary} by approximating $D$ as follows. Let $x_0\in \partial D$ be one of the finitely many points at which $\partial D$ fails to be smoothly embedded. In a first step, for $r>0$ small enough, one can cut out the corner $x_0$ by choosing suitable disjoint polygons in $D\cap B_r(x_0)\setminus B_{r/2}(x_0)$ whose length is of order $r$, resulting in a domain $D'\subset D$ with $D'\cap B_{r/2}(x_0) = \emptyset$ and $D'\setminus B_r(x_0) = D\setminus B_r(x_0)$. In a second step, one mollifies $\partial D'\cap B_r(x_0)$ by choosing local graph representations of the polygons. Doing so around each corner point $x_0\in \partial D$ results in a domain $D''$ depending on $r$ whose boundary is smooth. Now \cite[Theorem 1]{Brendle21JAMS} can be applied to $D''$. Letting $r$ go to zero simultaneously around all corner points, Hypothesis \eqref{eq:adm_boundary} implies that both sides of the inequality in \ref{item:22_new_4} converge.
\end{proof}

\subsection{The Gauss--Bonnet formula}
We recall the \emph{Gauss--Bonnet formula} (cf.\ \cite[Chapter 9]{Lee}). 
Suppose that $g$ is a Riemannian metric on $\Sigma$ with induced Riemannian measure $\mu$ and Gauss curvature $K$.
Then for each piecewise immersed and simply closed curve $\gamma\in PS(\mathbb{S}^1;\Sigma)$ satisfying $\gamma(\mathbb{S}^1)=\partial \Omega$ for an open set $\Omega\subset \Sigma$ one has 
\begin{equation}\label{eq:Gauss_Bonnet}
    \int_{\Omega} K \; \mathrm{d}\mu + \int_{\gamma} \kappa \; \mathrm{d}s + \sum_{i = 1}^k \theta_i = 2\pi.
\end{equation}

For a unit speed parametrization of $\gamma$, 
$\kappa(t) \vcentcolon = g\left(\frac{D}{\mathrm{d}t} \dot\gamma(t) , N(t)\right)$ is the \emph {geodesic curvature} of $\gamma$ at $t \in \mathbb{S}^1$, where 
$N(t)$ is chosen such that $(\dot\gamma(t), N(t))$ is an oriented basis of $T_{\gamma(t)}\Sigma$.
Furthermore if $\{a_1,...,a_n\}\subset \Sigma$ is the set of \emph{vertices of $\gamma$}, i.e.\ points with the property that $\gamma(t_i)=a_i$ for some $t_i\in \mathbb{S}^1$ with $\dot\gamma(t_i^+) \neq \dot\gamma(t_i^-)$, then $\theta_i \in [- \pi , \pi]$ is the \emph{exterior angle} of $\gamma$ at its vertices, meaning that
\begin{equation}
    \theta_i = \mathrm{sgn}\left(\mathrm{d}V(\dot\gamma(t_i^-), \dot\gamma(t_i^+)\right) \arccos\left( g(\dot\gamma(t_i^-), \dot\gamma(t_i^+))\right) ,
\end{equation}
where $\mathrm{d}V$ is a nonvanishing and  alternating 2-form determined by the orientation of $\Sigma$. 
We remark that the prerequisites on $\Omega$ in \eqref{eq:Gauss_Bonnet} are satisfied if and only if $\Omega$ is a topological disk, satisfies \eqref{eq:adm_boundary}, and has embedded boundary. In particular, nonembedded boundaries can not be treated with \eqref{eq:Gauss_Bonnet} at this stage.
An important consequence of \eqref{eq:Gauss_Bonnet} is the \emph{Gauss--Bonnet theorem} which says
\begin{equation}
    \int_\Sigma K \; \mathrm{d}\mu = 2\pi \chi(\Sigma),
\end{equation}
where $\chi(\Sigma)$ is the \emph{Euler characteristic} of $\Sigma$, a purely topological constant.

We now show how the previous results can be applied to prove \Cref{thm:main} under the additional assumption that the shortest closed geodesic is embedded. 
\begin{theorem}\label{thm:main_inj}
    Let $f\colon\mathbb S^2\to\R^n$ be an immersion with $\mathcal W(f)<6\pi$ and $\gamma\colon \mathbb S^1\to (\mathbb S^2,g_f)$ be a closed geodesic without self-intersections. Then
    \begin{equation}
        \mathcal L_{g_f}(\gamma) \ge C(6\pi - \mathcal W(f))\sqrt{\mathcal A(f)}
    \end{equation}
    for $C=\frac{1}{2\pi}\left(\sqrt{6\pi + \frac{2\pi}{4
    +\pi}}-\sqrt{6\pi}\right)$.
\end{theorem}
\begin{proof}
    We may assume $\mathcal A(f)=1$, after scaling. By the Jordan curve theorem, $\mathbb S^2\setminus\gamma(\mathbb S^1)$ has two connected components $D_1,D_2$, both of which are topological disks. After relabelling we have $\mathcal A(f,D_1)\ge 1/2$ and, for ${L}\vcentcolon=\mathcal L_{g_f}(\gamma)$, \Cref{lem:22_new}\ref{item:22_new_2}\ref{item:22_new_1} implies the existence of $x_0\in D_1$ such that 
    \begin{align}
        \dist(f(x_0),f(\partial D_1))&\ge \frac{1}{2}\Bigl(\diam(f(D_1)) - \frac{{L}}2\Bigr)\\
        &\ge \frac12\Bigl(\frac{1}{{L}+2\sqrt{6\pi}} -\frac{{L}}2\Bigr)=\vcentcolon a({L}).\label{eq:dist_ge_alpha}
    \end{align}
    Let $\bar{L}>0$ be such that $a(\bar{L})>0$. If ${L} < \bar{L}$, then, since $a$ is monotonically decreasing, we have that $a({L})\ge a(\bar{L})=\vcentcolon\alpha$, and \Cref{lem:22_new}\ref{item:22_new_3} combined with \eqref{eq:dist_ge_alpha} implies
    \begin{align}\label{eq:4pi_le_W(D_1)}
        4\pi \le \mathcal W(f,D_1) +2\int_{\partial D_1}\frac{\mathrm ds(x)}{|f(x)-f(x_0)|}\le \mathcal W(f,D_1) + \frac{2}{\alpha}{L}.
    \end{align}
    On the other hand, since $K \le \frac{1}{4}|H|^2$, \eqref{eq:Gauss_Bonnet} implies $2\pi \le \mathcal W(f,D_2)$. Using \eqref{eq:4pi_le_W(D_1)} it thus follows
    \begin{equation}
        {L} \ge \frac{\alpha}{2}(6\pi - \mathcal W(f)).
    \end{equation}
    If ${L}\ge \bar{L}$, then $\mathcal W(f)\ge 4\pi$ implies
    \begin{equation}
        {L} \ge \bar{L}~\frac{6\pi - 4\pi}{2\pi}\ge \frac{\bar{L}}{2\pi}(6\pi - \mathcal W(f)).
    \end{equation}
    The conclusion follows by choosing $\bar{L}>0$ such that $\frac{\bar{L}}{2\pi} = \frac{a(\bar{L})}{2}=\vcentcolon C$. 
\end{proof}

\section{The tiling induced by a closed geodesic and curvature estimates}\label{sec:tiling}

In this section, we will examine how a closed 
geodesic $\gamma\colon \mathbb{S}^1\to \Sigma$ with self-intersections divides a surface $\Sigma$ into \emph{tiles} bounded by geodesic segments. In particular, we will prove an upper bound on the total curvature in each of the resulting tiles if $\Sigma=\mathbb{S}^2$.

\begin{lemma}\label{lem:self_int_finite}
    Let $\gamma\colon \mathbb{S}^1\to (\Sigma,g)$ be a closed simply covered geodesic. Then $\gamma$ has only finitely many self-intersections all of which are nontangential.
\end{lemma}

\begin{proof}
     Since $\Sigma$ is compact, there exists $\varepsilon>0$ such that for all $x\in \Sigma$, we have that
     \begin{align}
         \exp_x \colon B_\varepsilon(0)\subset T_x\Sigma \to B^{d_g}_\varepsilon(x)
     \end{align}
     is a diffeomorphism, where $d_g$ is the Riemannian distance on $\Sigma$ induced by the Riemannian metric $g$. 
     Suppose that $(x,s,t)\in \Sigma\times \mathbb{S}^1\times \mathbb{S}^1$ with $x=\gamma(s)=\gamma(t), s\neq t$, and $(x',s',t')\in \Sigma\times \mathbb{S}^1\times \mathbb{S}^1$ is such that $d_g(x,x')$, $|s-s'|$, $|t-t'| <\varepsilon$. 
     By the definition of the exponential map, one checks that $\gamma(s')=\exp_x \left( \gamma'(s)(s'-s)\right)$ and similarly for $\gamma(t')$.
     By the local existence and uniqueness theorem for geodesics, all self-intersections must be non-tangential, so that $\gamma'(s)\neq \gamma'(t)$. We conclude that $x' = \gamma(s')=\gamma(t')$ if and only if $x=x'$, $t=t'$, and $s=s'$. It follows that the set 
     \begin{align}
         A \vcentcolon = \{ (x,s,t)\in \Sigma\times \mathbb{S}^1\times \mathbb{S}^1 \mid x=\gamma(s)=\gamma(t), s\neq t\}
     \end{align}
    consists of isolated points with a uniform lower bound on their distance. As $\Sigma\times \mathbb{S}^1\times \mathbb{S}^1$ is a compact metric space, it follows that $A$ is finite.
     %
\end{proof}

\begin{lemma}\label{lem:42}
Let $\Sigma$ be a topological sphere and let $\gamma\colon \mathbb{S}^1\to (\Sigma,g)$ be a closed simply covered  geodesic. Then $\Sigma\setminus \gamma(\mathbb
S^1)$ has finitely many connected components $D_i$, $i=1,\dots,m$, each of which is topologically an open disk satisfiying \eqref{eq:adm_boundary}. If $\gamma$ is embedded, then $m=2$ and $\int_{D_i}K\,\mathrm{d}\mu=2\pi$ for $i=1,2$. If $\gamma$ has self-intersections, then $m\geq 3$ and $\int_{D_i} K\,\mathrm{d} \mu < 2\pi$ for all $i=1,\dots,m.$
\end{lemma}
\begin{proof}
    By \Cref{lem:self_int_finite}, $\gamma$ can only have finitely many self-intersections, all of which must be nontangential. Hence, the number of connected components $D_1,\dots, D_m$ of $\Sigma\setminus \gamma(\mathbb{S}^1)$ is finite and each of the $D_i$ is open in $\Sigma$ with $\partial D_i$ given by a single curve $\gamma_i$ in $PS(\mathbb{S}^1;\Sigma)$ and thus satisfies \eqref{eq:adm_boundary}. Moreover, $\gamma_i\in C^0(\mathbb{S}^1;\Sigma)$ is simple closed and consists of geodesic segments of $\gamma$. By the Jordan curve theorem applied to $\gamma_i$, this implies that $D_i$ is topologically an open disk. 

    If $\gamma$ is embedded, then $m=2$ also follows from the Jordan curve theorem applied to $\gamma$, and in this case the Gauss--Bonnet formula yields
    \begin{align}
        \int_{D_i}K\,\mathrm{d}\mu = 2\pi, \qquad i=1,2.
    \end{align}
    
    If $\gamma$ has a self-intersection, then $m\geq 3$.     
    Moreover, the boundary $\partial D_i$ (taken inside $\Sigma$) of each disk $D_i$ contains a vertex which is a self-intersection point of $\gamma$. Indeed, otherwise, $\partial D_{i_0}$ is a submanifold of $\Sigma$ for some $i_0$, and, by the geodesic equation, it is open in $\gamma(\mathbb{S}^1)$. Then $\gamma^{-1}(\partial D_{i_0})$ is open and closed in $\mathbb{S}^1$, so $\gamma(\mathbb{S}^1)\subset D_{i_0}$ has no intersection point, a contradiction.
    Fix $1\leq i\leq m$ and let $D\vcentcolon= D_i$. We would now like to apply the Gauss--Bonnet formula \eqref{eq:Gauss_Bonnet} to $D$. However, the boundary $\partial D$ might have points of higher multiplicity, i.e.\
    not be parametrizable by  
    a simple closed curve, cf.\ \Cref{fig:D_nonemb_bdry}. Since $\partial D$ consists of parts of $\gamma$, by \Cref{lem:self_int_finite} it may only contain finitely many vertices, say $a_1, \dots, a_N\in \gamma(\mathbb{S}^1)$ with $N\in\N$.     
        Now, for each $1\leq j \leq N$ choose an open neighborhood $B_j$ of $a_j$ which is a topological disk and has embedded smooth boundary, such that $a_j$ is the only point of $\gamma$ in $B_j$ with higher multiplicity and such that $\bar B_j\cap \bar B_{j'}=\emptyset$ for all $j\neq j'$. Moreover, we may assume that $B_j\cap D$ is the disjoint union of finitely many open triangles $T_j^k$ with vertices at $a_j, b_j^k, c_j^k$ and corresponding exterior angles $\alpha_j^k, \beta_j^k, \gamma_j^k$, $1\leq k\leq d_j$, where two of the edges of $T_j^k$ consist of parts of $\gamma$ and the third edge is given by $\partial B_j \cap \partial T_j^k$, see \Cref{fig:geod_tri}.
    \begin{figure}
    \begin{subfigure}[t]{0.48\textwidth}
         \centering
        \def\svgwidth{1\linewidth}
        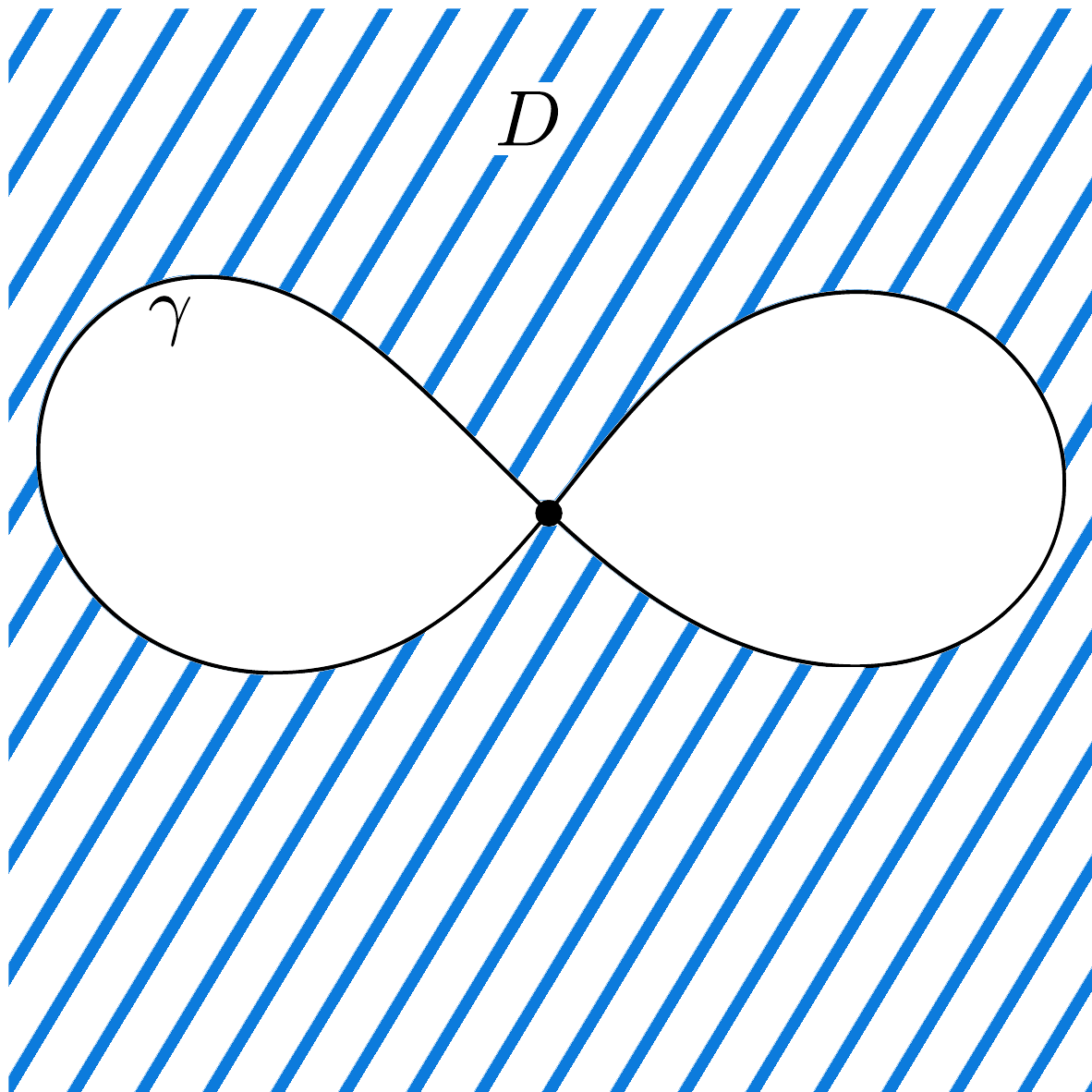
        \caption{A connected component $D$ whose boundary is not embedded.}
        \label{fig:D_nonemb_bdry}
     \end{subfigure}
    \hfill
    \begin{subfigure}[t]{0.48\textwidth}
         \centering
        \def\svgwidth{1\linewidth}
        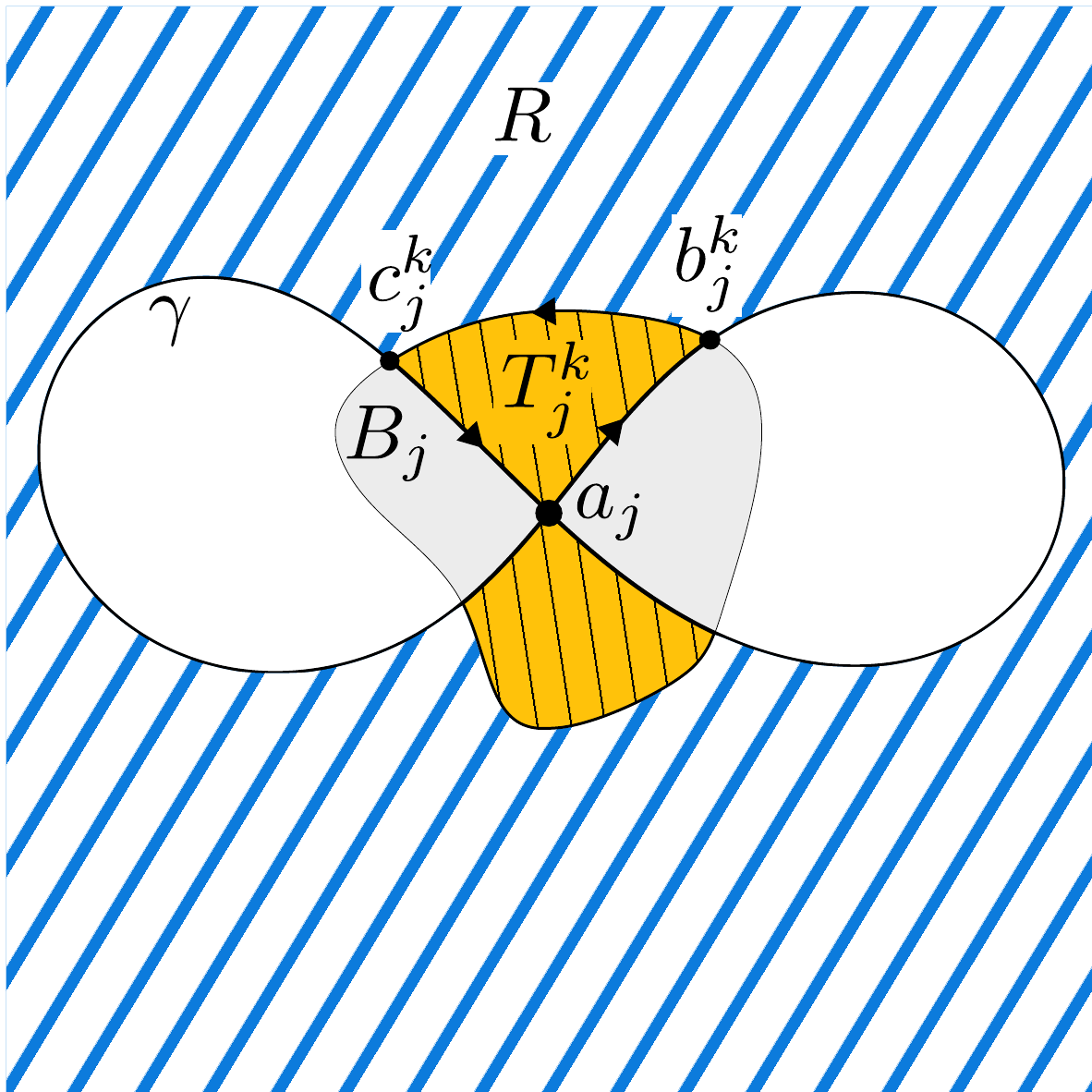
        \caption{Decomposing a neighborhood of $a_j$ into triangles. The set $R\subset \Sigma\cong \mathbb{S}^2$ is a topological disk.} 
        \label{fig:geod_tri}
     \end{subfigure}
     \caption{Removing a vertex from a connected component whose boundary has self-intersections.}
     \end{figure}
    For each $1\leq j\leq N$ and $1\leq k \leq d_j$, the Gauss--Bonnet formula yields
    \begin{align}\label{eq:GB_triangle}
        \int_{T_j^k} K \,\mathrm{d}\mu + \alpha_j^k +\beta_j^k+\gamma_j^k + \int_{\partial B_j \cap \partial T_j^k} \kappa \,\mathrm{d}s = 2\pi.
    \end{align}
    By the above cutting procedure, the set $R\vcentcolon= D\setminus (\bar B_1\cup \dots \cup \bar B_N)$ has
    (piecewise smooth) embedded and connected boundary. It follows that topologically $R$ is still an open disk, cf.\ \Cref{fig:geod_tri}. Indeed, $D$ retracts to $R$ which implies that $R$ is connected and simply connected, thus has Euler-characteristic one.    
    In particular, by construction $R$ satisfies the prerequisites of the Gauss--Bonnet formula \eqref{eq:Gauss_Bonnet}. Moreover, each of the vertices $b_j^k$, $c_j^k$ is also a vertex of $R$ with exterior angle $\hat\beta_j^k, \hat\gamma_j^k$, respectively, and these are the only vertices of $R$.
    The angles satisfy the relation
    \begin{align}\label{eq:angle_relation}
     \pi = \beta_j^k+\hat{\beta}_j^k = \gamma_j^k+\hat{\gamma}_j^k\quad  \text{ for } 1\leq k\leq d_j, 1\leq j\leq N.  
    \end{align}
    The Gauss--Bonnet formula for $R$ yields
    \begin{align}\label{eq:GB_R}
        \int_R K \,\mathrm{d}\mu  + \sum_{j=1}^N \sum_{k=1}^{d_j} (\hat{\beta}_j^k+\hat{\gamma}_j^k) + \sum_{j=1}^N \sum_{k=1}^{d_j}\int_{\partial B_j \cap \partial T^k_j} (-\kappa) \,\mathrm{d}s =  2\pi.
    \end{align}
    Note that the sign of the geodesic curvature integrals in \eqref{eq:GB_R} are opposite to those in \eqref{eq:GB_triangle}. Thus, summing up \eqref{eq:GB_triangle} and \eqref{eq:GB_R}, we find
    \begin{align} 
        \int_D K \,\mathrm{d}\mu &= \int_R K \,\mathrm{d}\mu + \sum_{j=1}^N\sum_{k=1}^{d_j} \int_{T_j^k}K \,\mathrm{d}\mu \\
        &= 2\pi - \sum_{j=1}^N\sum_{k=1}^{d_j} \left(\alpha^k_j+\beta^k_j+\gamma^k_j+\hat{\beta}^k_j + \hat{\gamma}_j^k -2\pi\right)\\
        & = 2\pi - \sum_{j=1}^N\sum_{k=1}^{d_j} \alpha^k_j,\label{eq:GB_D}
    \end{align}
    where we used \eqref{eq:angle_relation} in the last step. 
    The result now follows if we can prove that $\alpha_j^k$, the exterior angle of $T_j^k$ at $a_j$, satisfies $\alpha_j^k>0$ for all $1\leq k\leq d_j$, and $1\leq j\leq N$.    
    This follows by applying \Cref{lem:nagativeAngle} below to $U \vcentcolon =T_j^k$ and $p \vcentcolon = a_j$ and noting that $\alpha_j^k=0$ is impossible by \Cref{lem:self_int_finite}.
\end{proof}

\begin{remark}\label{rem:GB_with_selfint}
    For geodesic polygons with self-intersections at their boundaries,
    equation \eqref{eq:GB_D} can be viewed as a version of the Gauss--Bonnet theorem with an extended notion of exterior angles.
\end{remark}

\begin{remark}\label{rem:42_angle_control}
    Suppose that there exists
    $\xi\in (0,\pi/2)$ such that
	\begin{align}\label{eq:angle_control}
		\arccos\frac{g(\dot\gamma(t), \dot\gamma(s))}{|\dot\gamma(t)|_{g} |\dot\gamma(s)|_{g}} \in (\xi, \pi-\xi)\quad \forall t\neq s \in \mathbb{S}^1 \text{ with }\gamma(t)=\gamma(s).
	\end{align}
    Using this and $d_j\geq 1$ to estimate \eqref{eq:GB_D}, the estimate in \Cref{lem:42} can be improved to
	\begin{align}
			\int_{D_i}K\,\mathrm{d}\mu \leq 2\pi - N_i \xi,
	\end{align}
	where $N_i\in \N$ is the number of self-intersections of $\gamma$ which are contained in $\partial D_i$.
\end{remark}

\begin{lemma}\label{lem:nagativeAngle}
    Let $\gamma \colon \mathbb{S}^1 \rightarrow (\Sigma,g)$ be simply covered. Let $a_1<b_1<a_2<b_2\in [0,2\pi)$ and let $c_1 = \gamma\vert_{[a_1,b_1]}$, $c_2=\gamma\vert_{[a_2,b_2]}$ be embedded subsegments with $p \vcentcolon = \gamma(b_1)=\gamma(a_2)$ and $\gamma(x) \neq \gamma(y)$ for any $(x,y) \in [a_1,b_1) \times (a_2,b_2]$. Let $U \subset \Sigma$ be such that 
    \begin{enumerate}[label=(\roman*)]
        \item\label{item:L34_a} $U$ is an open triangle with embedded, piecewise smooth boundary $\partial U$;
        \item\label{item:L34_b} $\overline{U} \cap \gamma(\mathbb{S}^1) = \partial U \cap \gamma(\mathbb{S}^1) = c_1([a_1,b_1]) \cup c_2([a_2,b_2])$.
    \end{enumerate}
    Then the exterior angle  $\theta$ of $U$ at $p$ satisfies $\theta \geq 0$.
\end{lemma}


\begin{proof}
    Since the statement is local, after composing with a normal coordinate chart around the point $p$, we may assume that $\gamma$ is a planar curve, i.e.\ $U\subset \Sigma=\R^2$, $g=\langle\cdot,\cdot\rangle$, and $p=(0,0)\in\R^2$.
     Let $\theta$ be the exterior angle at $p=(0,0)$ as in the statement and assume, for the sake of contradiction, that $\theta <0$. Then the interior angle of $U$ is given by $\pi - \theta\in (\pi, 2\pi)$.

     First, we observe that there exists an open cone $C \subset \mathbb{R}^2$ of opening angle $\omega> \pi$ and some $\varepsilon>0$ such that
     \begin{align}\label{eq:contains_cone}
     C \cap B_\varepsilon(0) \subset U \cap B_\varepsilon(0).
     \end{align}
     This is readily checked using that the blow-up of the triangle $U$ around the vertex $p=(0,0)$ is a cone whose opening angle equals the interior angle $\pi-\theta > \pi$. From there it is easy to obtain an open cone $C$ with opening angle $\omega \in (\pi, \pi- \theta)$ which satisfies \eqref{eq:contains_cone}, see \Cref{fig:angle_neg}.

\begin{figure}
	\centering
	\def\svgwidth{0.25\linewidth}
	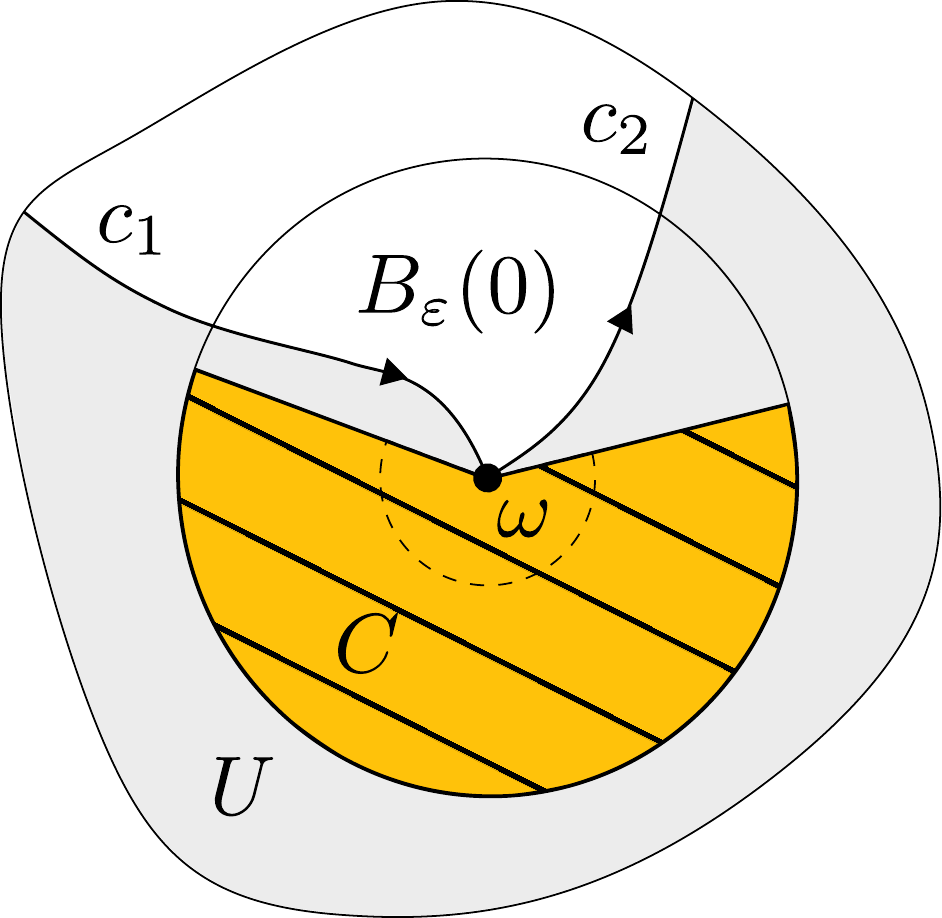
	\caption{Construction of the cone $C$ with opening angle $\omega\in (\pi, 2\pi)$.}
\label{fig:angle_neg}
\end{figure}  
By continuity, we have $\gamma(t)\in B_\varepsilon(0)$ for $t\in (b_1-\delta, b_1+\delta)$ for $\delta>0$ small enough. By assumption \ref{item:L34_b}, we have $\gamma(\mathbb{S}^1) \cap U = \emptyset$, and hence \eqref{eq:contains_cone} implies
    $\gamma(t)\in B_\varepsilon(0)\setminus C \text{ for all }t \in (b_1-\delta,b_1 + \delta).$

Now, $Z \vcentcolon= \R^2\setminus C$ is a closed cone with opening angle $2\pi-\omega<\pi$ and  $\gamma\vert_{(b_1-\delta,b_1+\delta)}$ is an immersion with $\gamma(t)\in Z$ for all $t\in (b_1-\delta,b_1+\delta)$.
After rotation, we may assume that $Z$ lies entirely in $\mathbb{H} \cup \{0 \}$, where $\mathbb{H} = \{ (x,y) \in \mathbb{R}^2 \mid y> 0\}$ is the open upper half plane. Write $\gamma = (\gamma^1,\gamma^2)$. As 
$\gamma^2$
becomes minimal at $t=b_1$, we have $\dot\gamma^1(b_1) = \lambda$ for some $\lambda \neq 0$, since $\gamma$ is an immersion. Taylor's expansion yields $\gamma(t) = \gamma(b_1) + (t-b_1)\lambda e_1  + r(t)(t-b_1)$ with $\lim_{t\to b_1}r(t)=0$. We thus find that $(t-b_1)\lambda e_1 + r(t)(t-b_1) \in Z$. As $Z$ is homothety-invariant and closed, 
we may take any $t> b_1$, divide by $t-b_1$ and pass to the limit $t\to b_1$  
to conclude $\lambda e_1 \in Z$, a contradiction.
 \end{proof}

\section{Proof of the main result}

The preceding purely intrinsic discussion now allows us to give a lower bound on the Willmore energy on the complement of each connected component of $\Sigma\setminus \gamma(\mathbb{S}^1)$ if $\Sigma=\mathbb{S}^2$ and $\gamma$ is a closed geodesic.

\begin{lemma}\label{lem:W_in_S_Di}
     Let $f\colon \mathbb{S}^2\to\R^n$ be an immersion. Let $\gamma\colon \mathbb{S}^1 \rightarrow (\mathbb{S}^2,g_f)$ be a simply covered closed geodesic
      and let $D_1,...,D_m$ be the 
      connected components of $\mathbb{S}^2 \setminus \gamma ( \mathbb{S}^1)$ as in \Cref{lem:42}. Then for all $i=1,\dots,m$ the following holds.
      \begin{enumerate}[label=(\roman*)]
      	\item If $\gamma$ is embedded, then $\mathcal{W}(f,\mathbb{S}^2\setminus D_i)\geq 2\pi$;
 	    \item If $\gamma$ has self-intersections, then $\mathcal{W}(f, \mathbb{S}^2\setminus D_i)> 2\pi$;
 	    \item If $\gamma$ has self-intersections with angles at the self-intersections bounded from below by $\xi>0$ in the sense of \eqref{eq:angle_control}, then we have $\mathcal{W}(f,\mathbb{S}^2\setminus D_i)\geq 2\pi+ N_i\xi$, where $N_i\in \N$ is the number of self-intersections of $\gamma$ contained in $\partial D_i$. 
      \end{enumerate}
 \end{lemma}

 \begin{proof}
    The pointwise estimate $K \leq \frac{1}{4} |H|^2$ and the global Gauss--Bonnet theorem imply 
 	\begin{align}\label{eq:intKS_Di}
 		\mathcal{W}(f,\mathbb{S}^2\setminus D_i)\geq \int_{\mathbb{S}^2\setminus D_i}K\,\mathrm{d}\mu= 4\pi -\int_{D_i}K\,\mathrm{d}\mu,
 	\end{align}
    as $\chi(\mathbb{S}^2)=2$.
    By \Cref{lem:42}, we conclude $\mathcal{W}(f,D_i)\geq 2\pi$ with strict inequality if $\gamma$ has self-intersections, thus (i) and (ii) follow. The last statement follows from estimating $\int_{D_i}K\,\mathrm{d}\mu$ by the upper bound in \Cref{rem:42_angle_control}.
 	%
 \end{proof}

With this tool, we can finally prove our main result.

\begin{proof}[Proof of Theorem \ref{thm:main}]
Let $f\colon \mathbb{S}^2\to\R^n$ be an immersion with $\mathcal{W}(f)<6\pi$ and let $\gamma\colon \mathbb{S}^1\to (\mathbb{S}^2, g)$ be a geodesic of length $L\vcentcolon=\mathcal L_g(\gamma)$ where $g= g_f$. We may assume that $\gamma$ is simply covered.
Let $D_1, \dots, D_m$ be the 
connected components of $\mathbb{S}^2\setminus\gamma(\mathbb{S}^1)$ as in \Cref{lem:42}. In particular, each $D_{{i}}$ is a topological disk satisfying \eqref{eq:adm_boundary}.
After rescaling, we may assume $\mathcal{A}(f)=1$.
Since counting the boundaries of the domains counts each edge twice, we have 
 \begin{equation}\label{eq:handshaking}
     \sum_{{{i}} = 1}^{m} \mathcal{L}_{g}(\partial D_{{i}})  = 2\mathcal{L}_{g}(\gamma).
 \end{equation}
Let $M\vcentcolon= \max_{{{i}}=1,\dots,m} \mathcal{A}(f,D_{{i}})$. Since $\mathbb R^3$ is isometrically embedded in $\mathbb R^4$, we may assume $n\ge4$ and apply \Cref{lem:22_new}\ref{item:22_new_4} as well as Equation \eqref{eq:handshaking} to deduce
\begin{align}
    1&=\sum_{i=1}^m\mathcal A(f,D_i)\le 2C_n^2\sum_{i=1}^m\Biggl(\Bigl(\int_{D_i}|H|\,\mathrm d\mu\Bigr)^2+4\mathcal L_g(\partial D_i)^2\Biggr)\\
    &\le 2C_n^2\sum_{i=1}^m\Bigl(4\mathcal W(f,D_i)\mathcal A(f,D_i) + 4L\mathcal L_g(\partial D_i)\Bigr) \le 8C_n^2(6\pi M + 2L^2).
\end{align}
Thus, assuming
\begin{equation}\label{eq:length_assumption}
    L^2 < 1/c_n\qquad\text{for $c_n\vcentcolon=16C_n^2$},
\end{equation}
we infer
\begin{equation}
    M \ge\frac{1-c_nL^2}{d_n}>0\qquad \text{for $d_n\vcentcolon=48\pi C_n^2$.}
\end{equation}
Let ${i_0}\in \{1, \dots,m\}$ be such that $\mathcal{A}(f,D_{{i_0}})=M$.
We apply Lemma \ref{lem:22_new}\ref{item:22_new_1}\ref{item:22_new_2} to obtain $x_0\in D_{i_0}$ with
\begin{align}
    &\dist(f(x_0),f(\partial D_{i_0}))\ge \frac12\Bigl(\diam(f(D_{i_0})) - \frac{\mathcal L_g(\partial D_{i_0})}2\Bigr)\\
    &\qquad\ge\frac{1}2\Bigl(\frac{2M}{\mathcal L_g(\partial D_{i_0}) + 2\sqrt{\mathcal W(f,D_{i_0})\mathcal A(f,D_{i_0})}} - \frac{\mathcal L_g(\partial D_{i_0})}{2}\Bigr)\\
    &\qquad\ge \frac{1-c_nL^2}{d_n(L + 2\sqrt{6\pi})} - \frac{L}{4}=\vcentcolon a(L). \label{eq:dist_to_boundary}
\end{align}
Let $\bar{L}>0$ be such that $a(\bar{L})>0$. If $L<\bar{L}$, then, since $a$ is monotonically decreasing, we have that $a(L)\ge a(\bar{L})=\vcentcolon\alpha$, and Lemma \ref{lem:22_new}\ref{item:22_new_3} combined with \eqref{eq:dist_to_boundary} implies
 \begin{equation}
     4 \pi \leq \mathcal{W}(f, {D_{{i_0}}}) + 2 \int_{\partial D_{{i_0}}} \abs{f(x)-f(x_0)}^{-1}  \; \mathrm{d} s(x) \leq \mathcal{W}(f,{D_{{i_0}}}) +  \frac{2}{\alpha}L.
 \end{equation}
 Applying \Cref{lem:W_in_S_Di} we find that $\mathcal{W}(f,{\mathbb{S}^2\setminus D_{{i_0}}})\geq 2\pi$ and consequently
\begin{align}\label{eq:6pi_cotradiction}
    L\ge \frac{\alpha}{2}(6\pi-\mathcal W(f)).
\end{align}
If $L\ge \bar{L}$, then $\mathcal W(f)\ge 4\pi$ implies
\begin{align}
    L\ge \bar{L}~\frac{6\pi-4\pi }{2\pi} \geq \frac{\bar{L}}{2\pi}(6\pi-\mathcal{W}(f)).
\end{align}
The conclusion follows
if there exists a positive solution $L_0$ of the equation $\frac{a({L_0})}{2} = \frac{{L_0}}{2\pi}$ satisfying $L_0^2<1/c_n$, cf.\ \eqref{eq:length_assumption}.
In this case, we may take $\bar{L}\vcentcolon = L_0$ above and thus \Cref{thm:main} follows with $C(n)\vcentcolon = \frac{a(\bar{L})}{2} = \frac{\bar{L}}{2\pi}>0$.\newline\indent 
The condition $\frac{a(L_0)}{2} = \frac{L_0}{2\pi}$ is equivalent to $L_0$ being the unique positive solution of the quadratic equation
\begin{equation}
    {L_0}^2 + 2L_0\alpha_1 = \frac{1}{\alpha_2}
\end{equation}
where 
\begin{equation}
    \alpha_2\vcentcolon=d_n\frac{4+\pi}{4\pi} + c_n=4(16+3\pi)C_n^2,\quad \alpha_1\vcentcolon=\frac{\sqrt{6\pi}d_n\frac{4+\pi}{4\pi}}{\alpha_2} = \frac{3\sqrt{6\pi}(4+\pi)}{16+3\pi}.
\end{equation}
Thus,
\begin{equation}
    L_0 = -\alpha_1 + \sqrt{\alpha_1^2 + \frac{1}{\alpha_2}}
\end{equation}
and one readily verifies \eqref{eq:length_assumption}.
\end{proof}

\begin{remark}\label{rem:opt_C}
    The above proof shows that in \Cref{thm:main} we may take
    \begin{align}
        C(n) = \frac{\sqrt{216\pi(4+\pi)^2 + \frac{16+3\pi}{C_n^2}} -6(4+\pi)\sqrt{6\pi} }{4\pi(16+3\pi)},
    \end{align}
    where
    $C_n = \min\{\sqrt{\frac{n-2}{32\pi}},\frac{75}{\sqrt{\pi}}\}$ is as in \Cref{lem:22_new}\ref{item:22_new_4} with $C_3 \vcentcolon = C_4$. Note in particular that $C(n)$ is bounded away from zero uniformly in $n$. \newline \indent
    By explicit computation, we see that $C(n)< C$ for all $n\geq 3$ where $C>0$ is as in \Cref{thm:main_inj}, i.e.\ the explicit constant in the embedded case is better.
\end{remark}

In the case of geodesics with self-intersections and bounded intersection angle, the above proof yields the following improvement of \Cref{thm:main}.
\begin{cor}
 Let $f\colon \mathbb S^2\to\R^n$ be an immersion, $\gamma\colon \mathbb S^1\to(\mathbb S^2,g_f)$ be a closed geodesic with self-intersections, and $\xi\in (0,\pi/2)$. Assume that $\gamma$ intersects itself in angles bounded by $\xi$ in the sense of (3.5). Then
$$ \mathcal{L}_{g_f}(\gamma)\geq C(n)(6\pi+\xi-\mathcal{W}(f)) \sqrt{\mathcal{A}(f)}.$$    
\end{cor}

\section{Closed geodesics on surfaces of revolution} \label{sec:spheroids}
Finding closed geodesics on a spheroid is a classical problem, see for instance \cite[Section 3.5]{Klingenber1995}, and the references therein. We are particularly interested in constructing closed geodesics on arbitrarily thin pieces of a spheroid having any given number of self-intersections, see \Cref{lem:spheroids} and \Cref{fig:spheroid_plots} below. To that end, we follow the approach in \cite{Alexander} which reduces the task of determining the number of intersections and the closedness of the curve to solving a suitable integral equation, see \Cref{lem:geod_surf_revol}.


Consider a general surface of revolution parametrized by
\begin{equation}\label{eq:ap:surface_of_revolution}
    f(u_1,u_2) = (h(u_2)\cos(u_1),h(u_2)\sin(u_1),g(u_2))
\end{equation}
where $h,g$ are smooth real-valued functions with $h>0$ and $u_1\in \R$, $u_2\in J$, where $J\subset \R$ is an interval.
Let $\gamma \vcentcolon= \sqrt{(h')^2 + (g')^2}$. Then the coefficients of the metric tensor are given by 
\begin{equation}\label{eq:ap:metric}
    E = \langle \partial_1 f, \partial_1f \rangle = h^2, \qquad F = \langle \partial_1f, \partial_2f\rangle = 0,\qquad G = \langle \partial_2f,\partial_2f\rangle = \gamma^2.
\end{equation}
For a curve with coordinate functions $u_1(t), u_2(t)$, 
the geodesic equations are
\begin{align}\label{eq:ap:geodesic1}
    \ddot u_1 + 2\frac{h'(u_2)}{h(u_2)}\dot u_1\dot u_2 & = 0,  \\ \label{eq:ap:geodesic2}
    \ddot u_2 - \frac{h(u_2)h'(u_2)}{\gamma(u_2)^2} \dot u_1\dot u_1 + \frac{\gamma'(u_2)}{\gamma(u_2)} \dot u_2 \dot u_2 & = 0,
\end{align}
where the dot derivative is with respect to the time variable $t$, see \cite[p.~5]{Alexander}. The first equation is $(\dot u_1h(u_2)^2)^{\bigcdot} =0$. Thus, $\dot u_1 h(u_2)^2 =c$ for some constant $c$, which is known as 
\emph{Clairaut's principle}. A geodesic of unit speed satisfies
\begin{equation}\label{eq:ap:unit_speed}
    1=\dot u_1^2E + \dot u_2^2G = \dot u_1^2h^2 + \dot u_2^2\gamma^2.
\end{equation}
Thus, by Clairaut's principle (cf. \cite[p.~6]{Alexander}),
\begin{equation}\label{eq:ap:Clairaut}
    \dot u_1 = \frac{c}{h(u_2)^2},\qquad|\dot u_2| = \frac{1}{\gamma}\sqrt{1- \frac{c^2}{h(u_2)^{2}}}.
\end{equation}

\begin{lemma}\label{lem:geod_surf_revol}
   Let $f$ be as in
   \eqref{eq:ap:surface_of_revolution} with $J=(-\alpha,\alpha)$, $\alpha>0$. Suppose
   \begin{align}\label{eq:hyp_on_h}
   \begin{split}
     &h>0,\quad h(x)=h(-x),\quad h'(x)<0 \text{ for }x\in (0,\alpha),\\
     & h(0)=1, \quad \lim_{x\to\alpha-}h(x)=0.
    \end{split}
   \end{align}
    Let $c\in (0,1)$, and let $\zeta_c(t) = f(u_1(t),u_2(t))$ be the unique unit speed geodesic on the surface of revolution 
    parametrized
    by $f$ with $u_1(0)=u_2(0)=0$, $\dot u_1(0)=c$, and $\dot u_2(0)=\sqrt{1-c^2}>0$. Then the following holds.
    \begin{enumerate}[label=(\roman*)]
        \item\label{item:geod_max} $\zeta_c\colon \R\to f(\mathbb{S}^1\times J)$ exists globally, we have $\max_\R u_2 = u_2(t_0) \in (0,\alpha)$ for some unique $t_0>0$, and this strict global maximum is given by the unique $t_0>0$ with $u_2(t_0)>0$ and $h(u_2(t_0))=c>0$.
        \item\label{item:def_I_c} If, in addition, $g$ is odd and 
            \begin{equation}\label{eq:ap:def_I_c}
   I_c \vcentcolon= -2c\int_{c}^{h(0)}\frac{\gamma(h^{-1}(y))}{y\sqrt{y^2-c^2}}\frac{\mathrm dy}{h'(h^{-1}(y))} = (N+1)\pi,
        \end{equation}
        then $\zeta_c\vert_{[0,4t_0]}$ is closed and possesses exactly $N$ distinct self-in\-tersec\-tions.
        \item\label{item:geod_length_bound} The length of $\zeta_c$ satisfies $2(N+1)\pi c \leq \mathcal{L}(\zeta_c) \leq \frac{2(N+1)\pi}{c}.$
    \end{enumerate}
\end{lemma}

\begin{proof}
    After reparametrization of the curve $(h,g)$, we may assume that it is parametrized by arc length, i.e.\ $\gamma\equiv 1$,  while still satisfying  \eqref{eq:hyp_on_h} with a suitably modified interval $J$.
    Such a reparametrization will not affect the surface, the maximum of $u_2$, or the curve $\zeta_c$, and also $I_c$ is invariant under replacing $h$ with $h\circ \varphi$ with $\varphi(0)=0$ and $\varphi'>0$.

    Suppose that there exists a sequence $t_n\to t^*$ with $u_2(t_n)\to \alpha$. Since $h(\alpha)=0$, \eqref{eq:ap:unit_speed} and \eqref{eq:ap:Clairaut} yield a contradiction. Hence, $\sup u_2 <\alpha$, and thus $\zeta_c$ does not intersect $f(\R\times \{\alpha\})$ and, with a similar argument using that $h$ is even, neither $f(\R\times \{-\alpha\})$. Consequently $u_1,u_2$ and $\zeta_c$ exist globally, i.e.\ for all $t\in \R$.
    Now, \eqref{eq:ap:geodesic2} reads
    \begin{align}
        \ddot{u}_2 = h(u_2)h'(u_2)\dot{u}_1^2.
    \end{align}
    As $c>0$, we have $\dot{u}_1>0$ as long as $u_2\in (0,\alpha)$ by \eqref{eq:ap:Clairaut}, and thus, by \eqref{eq:hyp_on_h}, $u_2$ is strictly concave on $u_2^{-1}(0,\alpha)$.
    We take $$t_0=\sup\{T>0\mid \dot{u}_2(t)>0\text{ for all }t\in [0,T]\}\in (0,\infty].$$
    By monotonicity and since $\sup u_2<\alpha$, the limit $u_\mathrm{max} = \lim_{t\to t_0-} u_2(t) \in (0,\alpha)$ exists. Moreover, $\dot u_2$ is strictly 
    decreasing in $(0, t_0)$ and hence also $\lim_{t\to t_0-}\dot{u}_2(t)$ exists. Now, if $t_0 = \infty$ was true, then there would exist a sequence $t_n\to t_0=\infty$ with $\lim_{n\to\infty}\ddot{u}_2(t_n) =0$ so that \eqref{eq:ap:geodesic2} and \eqref{eq:ap:Clairaut}
    would imply $h'(u_\mathrm{max})=0$ which by \eqref{eq:hyp_on_h} contradicts $u_\mathrm{max}>0$. 
    Thus $t_0<\infty$. By continuity $\dot{u}_2(t_0)=0$, and thus by concavity $u_2$ attains its strict global maximum $u_\mathrm{max}=u_2(t_0)>0$ at $t_0$, satisfying $h(u_2(t_0))=c$ by \eqref{eq:ap:Clairaut}. On the other hand, any $t_0>0$ with $u_2(t_0)>0$ and $h(u_2(t_0))=c>0$ necessarily satisfies  $\dot{u}_2(t_0)=0$ by \eqref{eq:ap:Clairaut}, and the proof of \ref{item:geod_max} is complete.

     For \ref{item:def_I_c}, define $I_c$ as above and let $t_0\in \R$ as in \ref{item:geod_max}. By the computation in \cite[p.~6]{Alexander}, we have
\begin{equation}\label{eq:ap:u_1-increment}
    u_1(t_0) - u_1(0) = \frac{I_c}{2} = \frac{(N+1)\pi}{2}.
\end{equation}
Now, for $t\in [t_0, 2t_0]$ we claim that
    \begin{align}
        \begin{split}\label{eq:ap:u_1u_2Extend}    
        u_1(t) &= -u_1(2t_0-t)+(N+1)\pi, \\
        u_2(t) &= u_2(2t_0-t). 
        \end{split}
    \end{align}
    This can be seen by verifying that the right hand side defines the component functions of a geodesic, i.e.\ a solution to \eqref{eq:ap:geodesic1}--\eqref{eq:ap:geodesic2}, which closes in a $C^1$-fashion with $\zeta_c$ at $t=t_0$ so \eqref{eq:ap:u_1u_2Extend} follows from the uniqueness of solutions to the geodesic equations. 
    We now discuss the self-intersections of $\zeta_c$ in $[0,2t_0]$. Using that $u_2$ is strictly increasing on $[0, t_0)$, we find that $\zeta_c\vert_{[0,t_0]}$ and $\zeta_c\vert_{[t_0,2t_0]}$ are injective. Now, suppose that we have $s\in [0, t_0)$ and $s'\in [t_0, 2t_0]$ such that $u_1(s)=u_1(s') \mod 2\pi$ and $u_2(s)=u_2(s')$. By \eqref{eq:ap:u_1u_2Extend} and since $\dot u_2>0$ on $[0,t_0]$, we conclude $s'=2t_0-s$ and it follows that $2u_1(s)= {(N+1)\pi} \mod 2\pi$ or equivalently
    \begin{align}\label{eq:u_1_solutions_mod_2pi}
     2u_1(s) =\begin{cases}
         \pi \mod 2\pi  & \text{if $N$ is even,}\\
         0 \mod 2\pi & \text{if $N$ is odd}.
     \end{cases}
     \end{align}
     Since $2u_1$ is strictly increasing from $0$ to $(N+1)\pi$ on $[0,t_0]$ by \eqref{eq:ap:u_1-increment}, if $N$ is even, then \eqref{eq:u_1_solutions_mod_2pi} has exactly $N/2$ solutions in $[0, t_0)$.
    If $N$ is odd, then there are exactly $(N+1)/2$ solutions of \eqref{eq:u_1_solutions_mod_2pi} in the interval $[0,t_0)$, one of which is $s=0$. Using that $h$ is an even function and $g$ is odd, we may argue as in \eqref{eq:ap:u_1u_2Extend} to show that for $t\in [2t_0, 4t_0]$ we have
    \begin{align}
        u_1(t) &= u_1(t-2t_0) + (N+1)\pi, \\
        u_2(t) &= -u_2(t-2t_0). 
    \end{align}
    This implies that $\zeta_c$ is $C^1$-closed on $[0, 4t_0]$ and statement 
    \ref{item:def_I_c} is proven.
    
    For the length bound \ref{item:geod_length_bound}, observe that by \eqref{eq:ap:Clairaut} and \eqref{eq:ap:u_1-increment} we have
    \begin{align}
        \frac{(N+1)\pi}{2} = u_1(t_0)-u_1(0) =\int_0^{t_0} \frac{c}{h(u_2(t))^2}\mathrm{d}t.
    \end{align}
    Since $h(u_2)$ is montonotically decreasing on $[0,t_0]$ with $h(u_2(0))=1$ and $h(u_2(t_0))=c$, we conclude 
    $
        ct_0 \leq {(N+1)\pi/2} \leq t_0/c.
    $
    Using that $\zeta_c$ is parametrized by arc length, we have $\mathcal{L}(\zeta_c)=4t_0$ and \ref{item:geod_length_bound} follows.
\end{proof}

\begin{figure}
     \centering
     \hfill
     \begin{subfigure}[b]{0.4\textwidth}
         \centering
         \includegraphics[width=\textwidth]{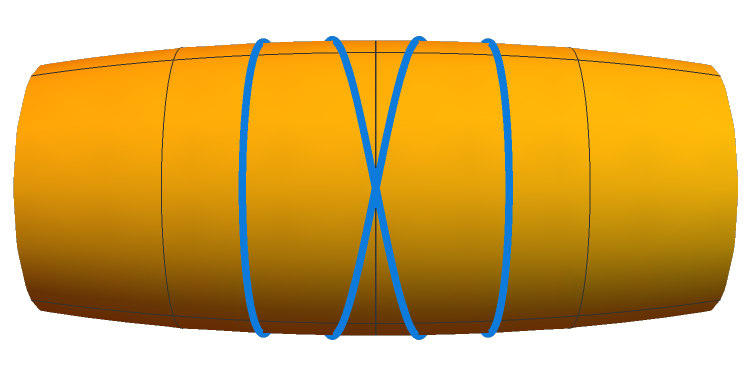}
         \caption{$b\simeq 4.038$, $c\simeq0.980$, $\varepsilon = 0.2$, $N=3$.}
     \end{subfigure}
     \hfill
     \begin{subfigure}[b]{0.4\textwidth}
         \centering
         \includegraphics[width=\textwidth]{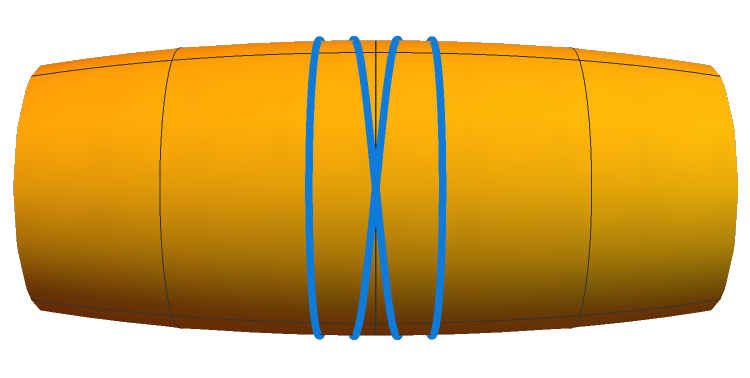}
          \caption{$b\simeq 4.009$, $c\simeq 0.995$, $\varepsilon = 0.1$, $N=3$.}
     \end{subfigure}
     \hfill\\
     \medskip
     \hfill
     \begin{subfigure}[b]{0.4\textwidth}
         \centering
         \includegraphics[width=\textwidth]{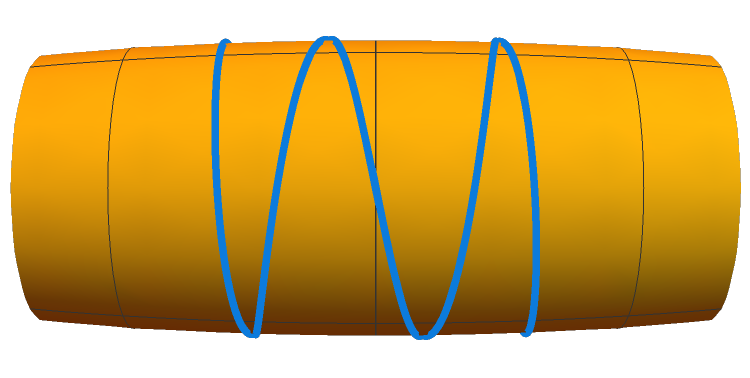}
         \caption{$b\simeq 5.049$, $c\simeq0.980$, $\varepsilon = 0.2$, $N=4$.}
     \end{subfigure}
     \hfill
     \begin{subfigure}[b]{0.4\textwidth}
         \centering
         \includegraphics[width=\textwidth]{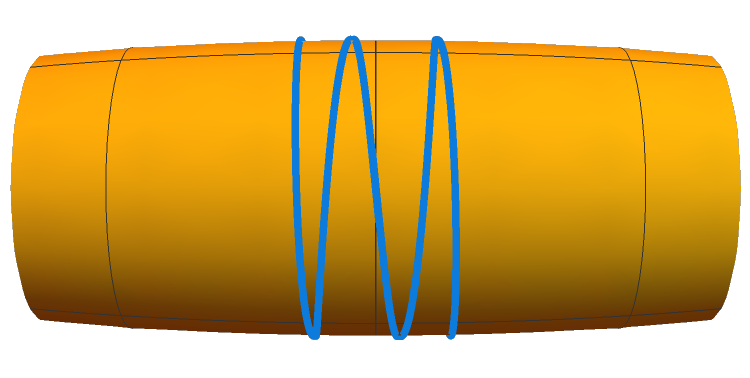}
         \caption{$b\simeq 5.012$, $c\simeq 0.995$, $\varepsilon = 0.1$, $N=4$.}
     \end{subfigure}
     \hfill
     \caption{Plots of closed geodesics with self-intersections on the spheroid for different parameters $b$, $c$, $\varepsilon$, $N$.}
     \label{fig:spheroid_plots}
 \end{figure}

\begin{lemma} \label{lem:spheroids}
    For all $N\in\N$ and $\varepsilon > 0$ there exists ${N}+1<b<{N}+1+\varepsilon$ such that the surface of revolution parametrized by $f$ as in \eqref{eq:ap:surface_of_revolution} with $(h,g)=(\cos,b\sin)$ and $J=(-\frac{\pi}{2},\frac{\pi}{2})$ contains a closed geodesic $\zeta_c =f(u_1,u_2)$ with exactly $N$ distinct self-intersections, $\max_{\R} |u_2| <\varepsilon$, and
    \begin{align}\label{eq:spheroid_length_geod}
        2(N+1)\pi (1-\varepsilon) \leq \mathcal{L}(f\circ \zeta) \leq \frac{2(N+1)\pi}{1-\varepsilon}.
    \end{align}
\end{lemma}

\begin{proof}
    Let $\zeta_c$, $u_1$, $u_2$ be as in \Cref{lem:geod_surf_revol}. Clearly, $h=\cos$ satisfies \eqref{eq:hyp_on_h}. By \Cref{lem:geod_surf_revol}\ref{item:geod_max}, the geodesic $\zeta_c(t)=f(u_1(t), u_2(t))$ exists globally and $\max_\R u_2(t_0)=h^{-1}(c) = \arccos c\to 0+$ as $c\to 1-$. The statement thus follows from \Cref{lem:geod_surf_revol}\ref{item:geod_max}--\ref{item:geod_length_bound} if we can prove that given $\varepsilon\in (0,1)$ and $N\in \N$, there exists $1-\varepsilon<c<1$ and $N+1<b<N+1+\varepsilon$ such that $I_c=(N+1)\pi$ with $I_c$ as in \eqref{eq:ap:def_I_c}.  
    By direct computation
    \begin{equation}
        h'(h^{-1}(y)) = -\sqrt{1-y^2},\qquad \gamma(x)^2 = h'(x)^2 + g'(x)^2 = 1 + (b^2-1)h(x)^2,
    \end{equation}
    which gives 
    \begin{align}
        I_c &= 2c\int_{c}^1\frac{\sqrt{1 + (b^2 - 1)y^2}}{y\sqrt{y^2 - c^2}}\frac{\mathrm dy}{\sqrt{1 - y^2}} \\
        &= 2c\int_c^1\frac{\sqrt{b^2-1 +y^{-2}}}{\sqrt{(y^2 - c^2)(1-y^2)}}\,\mathrm dy.\label{eq:Icformula}
    \end{align}
    The complete elliptic integral of the first kind is defined by
    \begin{equation}
        K(k) = \int_0^{\frac{\pi}{2}}\frac{\mathrm dx}{\sqrt{1-k^2\sin^2x}}\qquad \text{for $0\leq k <  1$,}
    \end{equation}
    see \cite[110.06]{ByrdFriedman}. Obviously, $K\geq K(0)= \pi/2$. By \cite[217.00]{ByrdFriedman}, as $0<c<1$ there holds 
    \begin{align} 
        I_c &< 2\sqrt{c^2(b^2-1) +1}\int_c^1\frac{\mathrm dy}{\sqrt{(y^2 - c^2)(1-y^2)}} \\
        &=2\sqrt{c^2(b^2-1) +1} \,K(\sqrt{1-c^2}).\label{eq:ap:lower_bound}
    \end{align}%
    Consider now
    \begin{equation}
        L(c)\vcentcolon = 2\sqrt{c^2((N+1)^2-1) +1} \,K(\sqrt{1-c^2})\qquad \text{for $0<c<1$.}
    \end{equation}
    There holds
    \begin{align}
        L'(c)&=\frac{2c((N+1)^2 -1)K(\sqrt{1-c^2})}{\sqrt{c^2((N+1)^2-1)+1}} \\
        &\quad + 2\sqrt{c^2((N+1)^2-1)+1}\,K'(\sqrt{1-c^2})\frac{-c}{\sqrt{1-c^2}}.
    \end{align}
    From the power series expansion of $K$, see \cite[900.00]{ByrdFriedman}, we infer $\lim_{k\to0+}K(k) = \pi/2$ as well as $\lim_{k\to0+}K'(k)/k = \pi/4$ which implies
    \begin{align}
        \lim_{c\to1-}L'(c) &= \frac{2((N+1)^2 - 1)\pi/2}{N+1} - 2(N+1)\pi/4 \\
        &= \pi \Bigl(N+1 - \frac{1}{N+1} - \frac{N+1}{2}\Bigr) > 0.
    \end{align}
    Since $L(1) = (N+1)\pi$, it thus follows $L(c_1)<(N+1)\pi$ for some $1-\varepsilon < c_1 < 1$. By \eqref{eq:ap:lower_bound} and a continuity argument, we can now choose $N+1 < b <N+1 +\varepsilon$ such that $I_{c_1}<(N+1)\pi$. On the other hand, \eqref{eq:Icformula} implies that
    \begin{equation}
        I_c>2cb\int_c^1\frac{\mathrm dx}{\sqrt{(x^2 - c^2)(1-x^2)}} = 2cbK(\sqrt{1-c^2}) \overset{c\to 1-}{\to} b\pi >(N+1)\pi.
    \end{equation}
    Hence, there exists $c_2$ with $c_1<c_2<1$ such that $I_{c_2}>(N+1)\pi$. Now, the conclusion follows from the intermediate value theorem.
\end{proof}

\section{Sphere inversion and spherical replacement}
\label{sec:optimality_genus}

In this section, we construct the embeddings $f_\varepsilon$ in Example \ref{ex:genus_zero} using a sphere inversion and an appropriate spherical replacement. The goal is that the inverted surface does not only look like a round sphere, but in fact contains a round sphere with a small cap removed. It is then not too difficult to find a short curve that we may evolve by the curve shortening flow. The avoidance principle  ensures that the curve does not shrink to a point, hence the flow exists globally and converges to a short geodesic.

Let first $f\colon \Sigma \to \R^n$ be any smooth embedding, and let $D\vcentcolon=\{z\in\R^2\mid |z|<1\}$ be the open unit disk. After a rigid motion of the ambient space, we can choose a parametrization $\phi\colon D\to\Sigma$ such that $f$ has the local graph representation 
\begin{equation}\label{eq:ap:graph}
    (f\circ \phi)(z) = (z,u(z))
\end{equation}
for some $u\in C^\infty(D;\R^{n-2})$ with $u(0)=0$ and $\mathrm Du(0)=0$. Choose a function $\eta\in C^\infty(\R;\R)$ such that 
\begin{equation}
    \eta(t) = 
    \begin{cases}
        1 & \text{for $t\geq2$}\\
        0 & \text{for $t\leq 1$}
    \end{cases}
\end{equation}
and 
\begin{equation}\label{eq:ap:cut-off}
    |\eta| + |\eta'| + |\eta''|<C
\end{equation}
for some universal constant $C<\infty$. For all $\delta > 0$ define
\begin{equation}
    \eta_\delta(z)\vcentcolon=\eta\Bigl(\frac{|z|}{\delta}\Bigr), \quad u_\delta(z) \vcentcolon =u(z)\eta_\delta(z)\qquad \text{for $z\in D$}
\end{equation}
and let $f_\delta \colon \Sigma \to \R^n$ be the immersion that results by replacing $u$ in \eqref{eq:ap:graph} with $u_\delta$ (for $\delta <1/2$). We are going to prove the following energy expansion.
\begin{lemma}\label{lem:replacement}
    There exists a universal constant $C<\infty$ such that
    \begin{equation}
        |\mathcal W(f) - \mathcal W(f_\delta)| \leq C \|\mathrm D^2u\|_{L^\infty(D)}\delta^2
    \end{equation}
    for $\delta>0$ small enough. 
\end{lemma}
\begin{proof}
    With $I_{2\times 2}$ denoting the $2\times 2$-identity matrix, we have
    \begin{align}
        \mathrm D\eta_\delta(z) &= \frac{1}{\delta}\eta'\Bigl(\frac{|z|}{\delta}\Bigr)\frac{z}{|z|},\\
        \mathrm D^2\eta_\delta(z) &=\frac{1}{\delta^2}\eta''\Bigl(\frac{|z|}{\delta}\Bigr)\frac{z}{|z|}\otimes \frac{z}{|z|} + \frac{1}{\delta|z|}\eta'\Bigl(\frac{|z|}{\delta}\Bigr)\Bigl(I_{2\times 2} - \frac{z}{|z|}\otimes \frac{z}{|z|}\Bigr)
    \end{align}
    and hence, by \eqref{eq:ap:cut-off},
    \begin{equation}\label{eq:ap:expansion_cut-off}
        |\mathrm D\eta_\delta| \leq \frac{C}{\delta},\qquad |\mathrm D^2\eta_\delta| \leq \frac{C}{\delta^2}.
    \end{equation}
    Since 
    \begin{align}
        \mathrm Du_\delta & = \eta_\delta \mathrm Du + u \mathrm D\eta_\delta,\\
        \mathrm D^2u_\delta & = \mathrm Du\otimes\mathrm D\eta_\delta +\mathrm D\eta_\delta \otimes \mathrm Du + \eta_\delta\mathrm D^2u + u\mathrm D^2\eta_\delta
    \end{align}
    and, by Taylor expansion, 
    \begin{equation}
        \frac{|u(z)|}{|z|^2} + \frac{|\mathrm Du(z)|}{|z|} \leq C\|\mathrm D^2u\|_{L^\infty(D)}
    \end{equation}
    it thus follows
    \begin{equation}
        |\mathrm Du_\delta| \leq C(|z|+\delta)\|\mathrm D^2u\|_{L^\infty(D)},\qquad |\mathrm D^2 u_\delta| \leq C\|\mathrm D^2u\|_{L^\infty(D)}.
    \end{equation}
    Denoting with $g$ and $g_\delta$ the metric tensors induced by $f$ and $f_\delta$, respectively, in the parametrization $\phi$ and using \cite[Lemma 2.1]{BauerKuwert}, we infer for $\delta>0$ with 
    \begin{equation}\label{eq:ap:smallness}
        |\mathrm Du(z)| + |\mathrm Du_\delta(z)|\leq 1\qquad\text{for all $z\in 2\delta D\vcentcolon= \{ 2\delta z \; | \; z \in D\}$}
    \end{equation}
    that
    \begin{align}
        &|\mathcal W(f) - \mathcal W(f_\delta)|=\bigl\||H|^2\sqrt{\det g} - |H_\delta|^2\sqrt{\det g_\delta}\bigr\|_{L^1(2\delta D)}\\
        &\quad\leq \bigl\||H|^2\sqrt{\det g} - |\Delta u|^2\bigr\|_{L^1(2\delta D)} + \bigl\||\Delta u|^2 - |\Delta u_\delta|^2\bigr\|_{L^1(2\delta D)} \\
        &\qquad + \bigl\||\Delta u_\delta|^2 - |H_\delta|^2\sqrt{\det g_\delta}\bigr\|_{L^1(2\delta D)}\\
        &\quad \leq C\Bigl(\bigl\||\mathrm Du||\mathrm D^2u|\bigr\|_{L^2(2\delta D)}^2 + \|\mathrm D^2u\|_{L^2(2\delta D)}^2 + \|\mathrm D^2u_\delta\|_{L^2(2\delta D)}^2 \\
        &\qquad  \qquad + \bigl\||\mathrm Du_\delta||\mathrm D^2u_\delta|\bigr\|_{L^2(2\delta D)}^2\Bigr) \\
        &\quad \leq C\|\mathrm D^2u\|_{L^\infty(2\delta D)}^2\delta^2.
    \end{align}
    Hence, the conclusion follows since \eqref{eq:ap:smallness} is satisfied for small $\delta>0$.
\end{proof}
Now let $\nu\in\{(0,0)\}\times \R^{n-2}$ be a vector of unit length. For all $\lambda>0$ define the Möbius transformation
\begin{equation}\label{eq:moebIlambda}
    I_\lambda\colon \R^n\setminus\{-\lambda\nu\}\to \R^n,\qquad I_\lambda(x) \vcentcolon = 2\lambda\frac{x + \lambda\nu}{|x + \lambda\nu|^2}
\end{equation}
and the unit sphere $$\mathbb S^2_\nu\vcentcolon=\{(z,0) +t\nu\mid z\in\R^2,\,t\in\R,\, |(z,0) + t\nu - \nu| =1\}\subset \R^n$$ with center $\nu$. Then the following holds.

\begin{lemma}\label{lem:cap}
    Suppose $\Sigma$ has genus $p\geq1$ and $f_\delta\colon \Sigma \to\R^n$ is defined as before \Cref{lem:replacement}. Then, for all $\delta,\lambda>0$ small enough, the surface $(I_\lambda\circ f_\delta)\colon\Sigma\to\R^n$  satisfies
    \begin{align}\label{eq:ap:cap}
        &(I_\lambda\circ f_\delta\circ \phi)(\delta D) = \mathbb S^2_\nu \cap \Bigl\{(z,\zeta)\in\R^2\times\R^{n-2}\mid |\zeta|>\frac{2\lambda^2}{\delta^2+\lambda^2}\Bigr\}, \\
        \label{eq:inE-minus}
        &\|I_\lambda\circ f_\delta\|_{L^\infty(\Sigma\setminus\phi(\delta D))} \leq \frac{2\lambda}{\delta}.
    \end{align}
    Moreover, the closed curve $\gamma_{\lambda,\delta}\vcentcolon=(I_\lambda\circ f_\delta\circ\phi)\colon\partial \delta D\to\R^n$ has length $\mathcal L(\gamma_{\lambda,\delta})=2\pi\frac{2\lambda\delta}{\delta^2 + \lambda^2} \leq 2\pi\frac{2\lambda}{\delta}$, is null-homotopic in $(I_\lambda\circ f_\delta\circ \phi)(\delta \bar D)$, but not null-homotopic in $(I_\lambda\circ f_\delta)\bigl(\Sigma\setminus\phi(\delta D)\bigr)$.
\end{lemma}
\begin{proof}
    First choose $\lambda,\delta>0$ small enough such that $f_\delta$ is an embedding, $-\lambda\nu\notin f_\delta(\Sigma)$, and
    \begin{equation}\label{eq:ap:dist}
        \dist\Bigl(f_\delta(\Sigma\setminus\phi(\delta D)),-\lambda\nu)\Bigr)\geq\delta.
    \end{equation}    
    Noting that $u_\delta = 0$ on $\delta D$ and that $I_\lambda\vert_{\R^2\times{0}^{n-2}}$ parametrizes the punctured sphere $\mathbb S^2_\nu\setminus\{0\}$, the decomposition
    \begin{equation}
        I_\lambda(z,0) = \frac{2\lambda (z,0)}{|z|^2+\lambda^2} + \frac{2\lambda^2\nu}{|z|^2+\lambda^2}\qquad\text{for all $z\in \R^2$}
    \end{equation}
    implies both, \eqref{eq:ap:cap} and $\mathcal L(\gamma_{\lambda,\delta})=2\pi\frac{2\lambda\delta}{\delta^2 + \lambda^2}$. Using \eqref{eq:ap:dist}, \eqref{eq:inE-minus} directly follows from the definition of $I_\lambda$.
    Clearly, $\gamma_{\lambda,\delta}$ is null-homotopic in 
    $(I_\lambda\circ f_\delta \circ \phi)(\delta \bar{D})$
    since it is a simple curve given by the boundary of the topological disk $(I_\lambda\circ f_\delta\circ\phi)( \delta D)$. If on the other hand, $\gamma_{\lambda,\delta}$ was also null-homotopic in $(I_\lambda\circ f_\delta)(\Sigma\setminus\phi( \delta D))$, then $\Sigma\setminus\phi( \delta D)$ would be  topologically a closed disk and thus, $\Sigma$ would be a topological sphere which is excluded by the hypothesis $p\geq1$.
\end{proof}

Finally, we are able to complete the construction in Example \ref{ex:genus_zero}.
Suppose $\Sigma$ has genus  $p \geq 1$. According to \cite{Simon,Kusner,BauerKuwert}, there exists a smooth immersion $f_0 \colon \Sigma \rightarrow \mathbb{R}^n$ with 
\begin{equation}
\mathcal{W}(f_0) = \min 
\{ \mathcal{W}(f) \mid f \colon \Sigma \rightarrow \mathbb{R}^n \;  \textrm{immersion} \} <8\pi.
\end{equation}
By \cite{LY}, $f_0$ is an embedding. We choose $\phi,u$ for $f_0$ as in \eqref{eq:ap:graph}, fix $\delta> 0$ small enough and look at the modified embedding $f_\delta$ of $f_0$ as in Lemma \ref{lem:replacement}. 
Consider for some $\nu \in \{(0,0) \} \times \mathbb{R}^{n-2}$ of unit length the Möbius transformation $I_\lambda$ as in \eqref{eq:moebIlambda}.
Now, $(I_\lambda \circ f_\delta)(\Sigma)$ is an embedded 2-dimensional submanifold of $\mathbb{R}^n$. 
Notice that for $\lambda >0$ small enough (depending on $\delta$) a neighborhood of the equator $G_\nu \vcentcolon= \partial D \times \{ \nu \} \subset \mathbb{S}_\nu^2$ lies in the image $(I_\lambda \circ f_\delta) ( \phi ( \delta D) ) $. In particular, $G_\nu$ is a geodesic on  $(I_\lambda \circ f_\delta)(\Sigma)$. 

 Recall now that for a 2-dimensional oriented manifold $M$ a smooth time-dependent family of closed immersed curves $\gamma \colon I \times \mathbb{S}^1 \rightarrow M$ is said to be an evolution by \emph{curve shortening flow} if
\begin{equation}
    \partial_t \gamma(t,x) =  \kappa(t,x) N(t,x) \quad \forall t \in I , x \in \mathbb{S}^1,
\end{equation}
where $\kappa(t,\cdot)$ denotes the geodesic curvature of $\gamma(t,\cdot)$ with respect to the oriented unit normal $N(t,\cdot)$. 
 
 \begin{lemma}\label{lem:CSF}
   Let $\delta,\lambda > 0$ be small enough as in Lemma \ref{lem:cap} and
   $\max\{\frac{2\lambda}{\delta}, \frac{2\lambda^2}{\lambda^2 + \delta^2} \} < \frac{1}{2}$.  Then the curve shortening flow on $(I_\lambda\circ f_\delta)(\Sigma)$ with initial datum $\gamma_{\lambda,\delta}$ exists for all times and converges to a closed null-homotopic geodesic $\sigma_{\lambda,\delta}$ of $(I_\lambda \circ f_\delta)(\Sigma)$, whose length is shorter than $2\pi \frac{2\lambda}{\delta}$.
 \end{lemma}
The proof of Lemma \ref{lem:CSF} relies on the \emph{avoidance principle} of the curve shortening flow in 2-dimensional Riemannian manifolds.  Its proof follows the lines of the proof of embeddedness-preservation in \cite[Section 3]{Gage}. We present the argument in Appendix \ref{app:aviodanceprinciple} for the convenience of the reader.
\begin{proof}[Proof of Lemma \ref{lem:CSF}.] 
   Let $\gamma_1 \colon [0,t_{max}) \times \mathbb{S}^1 \to (I_\lambda \circ f_\delta)(\Sigma) $ be a maximal evolution by curve shortening flow with initial datum $\gamma_1(0,\cdot)= \gamma_{\lambda,\delta}$.
 Define the two disjoint sets $E_+ \vcentcolon= \{ (z, \zeta) \in \mathbb{R}^2 \times \mathbb{R}^{n-2} \mid |\zeta| \geq 1 \}$ and  $E_- \vcentcolon= \{ (z, \zeta) \in \mathbb{R}^2 \times \mathbb{R}^{n-2} \mid |\zeta| < 1 \}$.
   Observe that by Lemma \ref{lem:cap} and our choice of parameters we have $\gamma_{\lambda,\delta}(\mathbb{S}^1) \subset E_-$  and $\partial (E_+ \cap (I_\lambda\circ f_\delta)(\Sigma)) =\partial (E_+ \cap \mathbb{S}_\nu^2) = G_\nu$. In particular, $\gamma_{\lambda,\delta}(\mathbb{S}^1)$ and $G_\nu$  are disjoint. 
    Applying the avoidance principle (Lemma \ref{lem:csf}) to $\gamma_1$ and the constant evolution $\gamma_2(t, \cdot) \equiv G_\nu$ (which adheres to the evolution law since $G_\nu$ is a geodesic), we infer that $\gamma_1(t,\cdot)$ can not intersect $G_\nu = \partial (E_+ \cap (I_\lambda \circ f_\delta)(\Sigma))$ in finite time. Hence the flow must stay in $\overline{E_-} \cap (I_\lambda \circ f_\delta)(\Sigma)$.
    From Lemma \ref{lem:cap} one can infer that  $\gamma_{\lambda,\delta}$ is not null-homotopic in $(I_\lambda \circ f_\delta )(\Sigma) \cap \overline{E_-}$. Thus we conclude from Grayson's theorem \cite[Theorem 0.1]{Grayson} that $t_{max}= \infty$ and $\gamma_1$ converges smoothly to a closed geodesic $\sigma_{\lambda,\delta}$, which is homotopic to $\gamma_1(0,\cdot)=\gamma_{\lambda,\delta}$, hence null-homotopic in $(I_\lambda \circ f_\delta)(\Sigma)$ by Lemma \ref{lem:cap}).
    Since the curve shortening flow decreases the length of curves, the asserted length bound follows immediately from Lemma \ref{lem:cap}. 
\end{proof}

Once this lemma is shown one can readily construct the immersions $f_\varepsilon$ in Example \ref{ex:genus_zero}.
\begin{proof}[Proof of the claim in Example \ref{ex:genus_zero}]
Let $\varepsilon> 0$ be fixed and $f_0$ be as above. By Lemma \ref{lem:replacement} one can choose $\delta = \delta(\varepsilon) > 0$ small enough such that the modified immersion $f_\delta$ constructed above is an embedding and satisfies $\mathcal{W}(f_\delta) \leq \mathcal{W}(f_0) + \varepsilon$. 
    By Lemma \ref{lem:CSF} one concludes that for each $\lambda > 0$ small enough there exists a closed null-homotopic geodesic $\sigma_{\lambda,\delta}$ of length less than $2\pi \frac{2\lambda}{\delta}$. Choosing $\lambda < \frac{\varepsilon\delta}{4 \pi}$ the claim follows.
\end{proof}

\appendix

\section{An example of a graph with unbounded  curvature} \label{sec:Toro-curvature}
In this section we analyze the Gauss curvature of the graph given in Example 1 of \cite{Toro}. 
\begin{lemma}\label{lem:Toro}
    The Gauss curvature of the graph of the function 
    \begin{equation}
        u\colon \R^2\to \R,\quad u(x,y)\vcentcolon = x\log |\log \sqrt{x^2+y^2}|
    \end{equation}
    satisfies
    \begin{equation}
        \lim_{x\to0} K(x,0) = \infty,\qquad \lim_{y\to0}K(0,y) = -\infty.
    \end{equation}
\end{lemma}
\begin{proof}
    Let $r\colon\R^2\to\R$ be defined by $r(x,y) \vcentcolon = \sqrt{x^2+y^2}$. On the open  unit disk $D = \{ r< 1 \}$
    direct computation gives
    \begin{equation}
        \partial_xu = \log|\log r| + \frac{x^2}{r^2\log r}, \qquad \partial_yu= \frac{xy}{r^2\log r}
    \end{equation}
    as well as 
    \begin{align}
        \partial^2_{xx}u &= \frac{x}{r^2\log r}\Bigl(3-\frac{2x^2}{r^2}-\frac{x^2}{r^2\log r}\Bigr),\\
         \partial^2_{xy}u &= \frac{y}{r^2\log r}\Bigl(1-\frac{2x^2}{r^2}-\frac{x^2}{r^2\log r}\Bigr),  \\
        \partial^2_{yy}u &= \frac{x}{r^2\log r}\Bigl(1-\frac{2y^2}{r^2}-\frac{y^2}{r^2\log r}\Bigr).
    \end{align}
    It follows that as $r \rightarrow 0+$ we have
    \begin{align}
        &|\mathrm Du(x,0)|^2=(\log|\log r|)^2 + o(1), && \det \mathrm D^2u(x,0) =\frac{1-\frac{1}{\log r}}{x^2(\log r)^2}\\ 
        &|\mathrm Du(0,y)|^2=(\log|\log r|)^2, &&\det \mathrm D^2u(0,y) = \frac{-1}{y^2(\log r)^2}.
    \end{align} 
    Using 
    that the Gauss curvature is given by $K = (\det \mathrm D^2u)/(1 + |\mathrm Du|^2)^2$, we infer
    \begin{equation}
        \lim_{x\to0}K(x,0) = -\lim_{y\to0}K(0,y) = \lim_{y\to0}\frac{1}{y^2(\log r)^2(\log|\log r|)^4} = \infty
    \end{equation}
    which completes the proof.
\end{proof}
\section{The avoidance principle on a surface}\label{app:aviodanceprinciple}
\begin{lemma}[Avoidance principle]\label{lem:csf}
    Let $M$ be a smooth compact two-dimensional oriented Riemannian manifold without boundary.  Consider for $T \in (0,\infty)$ two evolutions $\gamma_1,\gamma_2 \colon [0,T] \times \mathbb{S}^1 \rightarrow M$ of closed curves by the curve shortening flow such that $\gamma_1(0,\mathbb{S}^1) \cap \gamma_2(0, \mathbb{S}^1) = \emptyset$. Then $\gamma_1(t, \mathbb{S}^1) \cap \gamma_2(t, \mathbb{S}^1) = \emptyset$ for all $t \in [0,T]$.
\end{lemma}
\begin{proof}
By Nash's embedding theorem we may assume that $M \subset \mathbb{R}^n$ is an embedded submanifold. For $p \in M$ we denote by $\pi_{T_pM} : \mathbb{R}^n \rightarrow T_pM$ the orthogonal projection on $T_pM$. Recall that $|\cdot|$ denotes the Euclidean norm on $\mathbb{R}^n$. Suppose that $\eta= \eta(M) >0$ is chosen such that for all $p \in M$ the restriction  $\pi_{T_pM} \vert_{B_\eta(0) \cap (M-p)} \colon B_\eta(0) \cap (M - p) \rightarrow T_pM $ is injective and satisfies $|\pi_{T_p M} (v- p)| \geq \frac{1}{2} |v-p|$  for all $v\in B_\eta(p) \cap M$. For a vector $w \in T_p M$ the quantity $Jw = J_p w \in T_p M $ denotes the unique vector that is orthogonal to $w$, has the same norm as $w$, and satisfies that $(w,Jw)$ is an oriented basis of $T_p M$. A straightforward computation shows that there exists a  constant $C= C(M) >0$ such that for all $p,q \in M$ with $0< |p-q|< \eta$ one has
\begin{equation}
    \left| \frac{J_p \pi_{T_p M }(p-q)}{|\pi_{T_p M}(p-q)|} -  \frac{J_q \pi_{T_q M }(p-q)}{|\pi_{T_q M}(p-q)|} \right|  \leq C |p-q|.\label{eq:C(M)}
\end{equation}
   For $\mathbb{S}^1 = \mathbb{R} / \mathbb{Z}$ we define the time dependent Riemannian metric $\rho(t)$ on $\mathbb{S}^1 \times \mathbb{S}^1$ via $\rho(t)_{(x_1,x_2)}= |\partial_{x_1}\gamma_1(t,x_1)|^2 dx_1^2 + |\partial_{x_2}\gamma_2(t,x_2)|^2 dx_2^2$.
    We define now a distance function $d \colon [0,T] \times \mathbb{S}^1 \times \mathbb{S}^1 \rightarrow \mathbb{R}$ via 
 $
        d(t,x_1,x_2) = |\gamma_1(t,x_1)- \gamma_2(t,x_2)|^2.
 $ 
    We intend to show that $d> 0$ on $[0,T] \times \mathbb{S}^1 \times \mathbb{S}^1$. 
    \smallskip
    
    \textbf{Claim 1.}  $ \partial_t d - \Delta_{\rho(t)} d = - 4$. 
    
    \noindent\textit{Proof of Claim 1.} 
    Since $ \Delta_{\rho}  = \frac{1}{\sqrt{\mathrm{det}(\rho)}} \partial_{x_i} \left( \sqrt{\mathrm{det}(\rho)} \rho^{ij} \partial_{x_j} \right) $ one has
\begin{align}
    \Delta_{\rho(t)}  & = \frac{1}{|\partial_{x_1} \gamma_1(t,x_1)| | \partial_{x_2} \gamma_2(t,x_2)|} \Bigg( \partial_{x_1}\left( \frac{|\partial_{x_2}\gamma_2(t,x_2)|}{|\partial_{x_1} \gamma_1(t,x_1)|} \partial_{x_1}  \right) \\
    &\qquad\qquad\qquad\qquad\qquad\qquad\qquad+ \partial_{x_2}\left( \frac{|\partial_{x_1} \gamma_1(t,x_1)|}{|\partial_{x_2}\gamma_2(t,x_2)|} \partial_{x_2}  \right) \Bigg)  \\ & = \partial_{s_1s_1}^2  + \partial_{s_2s_2}^2,
\end{align}
    where $s_i$ is the (time dependent) arc length element of $\gamma_i(t,\cdot)$ with respect to $x_i$. 
    We compute
    \begin{align}
        \partial_{s_1} d &= 2 \langle \gamma_1(t,x_1) - \gamma_2(t,x_2) , \partial_{s_1} \gamma_1(t,x_1) \rangle,  \\\  \partial_{s_2} d &=  - 2 ( \gamma_1(t,x_1) - \gamma_2(t,x_2) , \partial_{s_2} \gamma_2(t,x_2) \rangle. 
    \end{align}
    Using that by the Frenet--Serret formulae on $M$ one has $\partial_{s_is_i}^2 \gamma_i(t,x_i) = \kappa_i(t,x_i) N_i(t,x_i)$ we find 
    \begin{align}
        \partial^2_{s_1s_1} d  & = 2 |\partial_{s_1} \gamma_1(t,x_1)|^2 + 2\langle \gamma_1(t,x_1) - \gamma_2(t,x_2) , \partial^2_{s_1s_1} \gamma_1(t,x_1) \rangle  \\ & = 2 + 2\langle \gamma_1(t,x_1) - \gamma_2(t,x_2) ,  \kappa_1(t,x_1)N_1(t,x_1) \rangle
    \end{align}
    and similarly
$
        \partial^2_{s_2s_2} d  
        = 2-  2 \langle \gamma_1(t,x_1) - \gamma_2(t,x_2) , \kappa_2(t,x_2) N_2(t,x_2) \rangle
    $
    so that 
    \begin{align}
        &\Delta_{\rho(t)}d \\
        &\quad = 4 + 2  \langle \gamma_1(t,x_1) - \gamma_2(t,x_2) , \kappa_1(t,x_1) N_1(t,x_1) - \kappa_2(t,x_2) N_2(t,x_2) \rangle  
        \\ &\quad  = 4 + \langle \gamma_1(t,x_1) - \gamma_2(t,x_2) , \partial_t \gamma_1(t,x_1) - \partial_t \gamma_2(t,x_2) \rangle = 4 + \partial_t d.
    \end{align}
    Claim 1 follows. For later use we also compute
    \begin{equation}\label{eq:secderdist}
        \partial^2_{s_1s_2} d = - 2 \langle \partial_{s_2}\gamma_2(t,x_2), \partial_{s_1} \gamma_1(t,x_1) \rangle .
    \end{equation}

For $\theta$, $\beta >0 $ we consider now the auxiliary function $h \colon [0,T] \times \mathbb{S}^1 \times \mathbb{S}^1 \rightarrow \mathbb{R}$ defined by $h(t,x_1,x_2) \vcentcolon= e^{\beta t} d(t,x_1,x_2) + \theta t$. 
 We may assume without loss of generality that for  $\theta > 0$ small enough one has $\inf_{t \in [0,T], x_1,x_2 \in \mathbb{S}^1} h < \eta$, since otherwise the lemma follows immediately from the observation that  $\theta \rightarrow 0+$ would yield $e^{\beta t} d(t,x_1,x_2) \geq \eta$ for all $t \in [0,T], x_1,x_2 \in \mathbb{S}^1$.
\smallskip

\textbf{Claim 2.} For $\beta > 2C^2$ with $C= C(M)$ as in \eqref{eq:C(M)} and for $\theta> 0$ small enough one has for all $t \in [0,T]$ and $x_1,x_2 \in \mathbb{S}^1$ that 
 $h(t,x_1,x_2) \geq \inf_{y_1,y_2 \in \mathbb{S}^1} h(0,y_1,y_2)$ . \\ 
 \textit{Proof of Claim 2.} Assume the opposite. Then there exist $\hat{t} \in (0,T]$ and $\hat{x}_1,\hat{x}_2 \in \mathbb{S}^1$ such that $h(\hat{t},\hat{x}_1,\hat{x}_2) = \min_{t \in [0,T], x_1,x_2 \in \mathbb{S}^1} h(t,x_1,x_2) (< \eta)$. We infer that at $(\hat{t},\hat{x}_1, \hat{x}_2)$ there holds   $d(\hat{t},\hat{x}_1,\hat{x}_2) < h(\hat{t},\hat{x}_1,\hat{x}_2) < \eta$ as well as $\partial_t h \leq 0$, $\partial_{s_1} h = \partial_{s_2} h = 0$, $\partial^2_{s_1s_1} h \geq 0 $, $\partial^2_{s_2s_2} h \geq 0 $ and $(\partial^2_{s_1s_1} h) (\partial^2_{s_2s_2} h) - (\partial^2_{s_1s_2} h)^2 \geq 0$ (as the latter expression is the determinant of the Hessian in $(\mathbb{S}^1 \times \mathbb{S}^1 , \rho(\hat{t}))$). 
Notice first that this implies that $d(\hat{t},\hat{x}_1,\hat{x}_2)> 0$, i.e.\ $\gamma_1(\hat{t},\hat{x}_1) \neq \gamma_2(\hat{t},\hat{x}_2)$. Indeed,  $d(\hat{t},\hat{x}_1,\hat{x}_2)= 0$ would imply $\partial_t d (\hat{t},\hat{x}_1,\hat{x}_2)= 0$ and with this one would infer  the contradiction $\partial_t h(\hat{t},\hat{x}_1,\hat{x}_2) = \theta > 0$. (We remark that in order to obtain $\partial_t d (\hat{t},\hat{x}_1,\hat{x}_2)= 0$ in the case of $\hat{t} = T$ one needs to extend the flow to $[0,T+ \varepsilon)$).
By \eqref{eq:secderdist} one has
\begin{align}
     \partial^2_{s_1s_2} h(\hat{t},\hat{x}_1,\hat{x}_2)  &= e^{\beta \hat{t}} \partial^2_{s_1s_2} d(\hat{t},\hat{x}_1,\hat{x}_2) \\
     &= - 2e^{\beta \hat{t}} \langle \partial_{s_2}\gamma_2(\hat{t},\hat{x}_2), \partial_{s_1} \gamma_1(\hat{t},\hat{x}_1)\rangle. \label{eq:secarcl}
 \end{align}
 
 Since also $\partial_{s_1} h(\hat{t},\hat{x}_1,\hat{x}_2) = \partial_{s_2} h(\hat{t},\hat{x}_1,\hat{x}_2) = 0$ we have 
 \begin{align}
     &\langle \partial_{s_1} \gamma_1(\hat{t}, \hat{x}_1) , \gamma_1(\hat{t},\hat{x}_1) - \gamma_2(\hat{t}, \hat{x}_2) \rangle \\
     &\qquad = \langle\partial_{s_2} \gamma_2(\hat{t}, \hat{x}_2) , \gamma_1(\hat{t},\hat{x}_1) - \gamma_2(\hat{t}, \hat{x}_2) \rangle = 0.
 \end{align}
Since $T_p M$ is 2-dimensional and $\gamma_1(\hat{t},\hat{x}_1) \neq \gamma_2(\hat{t},\hat{x}_2)$) for $i = 1,2$ this implies
\begin{equation}
    \partial_{s_i} \gamma_i(\hat{t},\hat{x}_i) =  \pm   J \frac{\pi_{T_{\gamma_i(\hat{t},\hat{x}_i)} M} (\gamma_1(\hat{t},\hat{x}_1) - \gamma_2(\hat{t}, \hat{x}_2))}{|\pi_{T_{\gamma_i(\hat{t},\hat{x}_i)} M} (\gamma_1(\hat{t},\hat{x}_1) - \gamma_2(\hat{t}, \hat{x}_2))|}. 
\end{equation}
Thereupon, using \eqref{eq:C(M)} we find 
\begin{equation}
    |\partial_{s_1} \gamma_1(\hat{t},\hat{x}_1) \pm \partial_{s_2}\gamma_2(\hat{t},\hat{x}_2)| \leq C \sqrt{d}.
\end{equation}
The polarization identity yields that 
$$
    |\langle \partial_{s_2}\gamma_2(\hat{t},\hat{x}_2), \partial_{s_1} \gamma_1(\hat{t},\hat{x}_1) \rangle | =  \frac{1}{2} \big| |\gamma_1(\hat{t},\hat{x}_1) \pm \gamma_2(\hat{t},x_2) |^2 - 2 \big|\geq 1- \tfrac{1}{2} C^2 d$$ 
and one infers from \eqref{eq:secarcl} that 
$
    |\partial^2_{s_1s_2} h | \geq (2 -  C^2 d) e^{\beta \hat{t}}.
$
Altogether one obtains (always evaluated at $(\hat{t},\hat{x}_1,\hat{x}_2)$)
\begin{align}\label{eq:ineq1}
    0 & \geq \partial_t h = \beta e^{\beta \hat{t}} d + e^{\beta \hat{t}} (\partial_t d) + \theta  
     = \beta e^{\beta \hat{t}} d + e^{\beta \hat{t}} (\Delta_{\rho(\hat{t})} d -4 ) + \theta.
\end{align}
Using that 
\begin{align}
     e^{\beta \hat{t}} (\Delta_{\rho(\hat{t})} d ) &= \partial^2_{s_1s_1}h + \partial^2_{s_2s_2} h \\
     &\geq 2 \sqrt{\partial_{s_1s_1}^2 h \partial^2_{s_2s_2} h} \geq   2 |\partial_{s_1s_2}^2 h| \geq (4 - 2C^2d)e^{\beta \hat{t}}
\end{align}
we find with \eqref{eq:ineq1} that
$
    0 \geq (\beta - 2C^2) e^{\beta \hat{t}} d + \theta \geq \theta > 0.  
$
A contradiction. Claim 2 follows. 

Looking at the limit case $\theta \rightarrow 0+$ in the statement of Claim 2 we infer that for a fixed $\beta > 2C^2$ there holds
\begin{equation}
    d(t,x_1,x_2) \geq  e^{-\beta T} \inf_{x_1,x_2 \in \mathbb{S}^1} d(0,x_1,x_2). 
\end{equation}
This finally proves the statement. 
\end{proof}

\section*{Acknowledgements}

This research was funded in whole, or in part, by the Austrian Science Fund (FWF), grant numbers \href{https://doi.org/10.55776/P32788}{10.55776/P32788} and \href{https://doi.org/10.55776/ESP557}{10.55776/ESP557}. Part of this work was conducted during the \emph{Surfaces22 summer school} in Frauenchiemsee, Germany, in August 2022. The authors are grateful to the organizers for the opportunity to attend and to the Joachim Herz Foundation (Project 800048) for covering the expenses.
Moreover, the authors would like to thank the referees for their valuable comments on the original manuscript.

\bibliography{Lib}

\begin{bibdiv}
\begin{biblist}

\bib{Alexander}{article}{
      author={Alexander, James~C.},
       title={Closed geodesics on certain surfaces of revolution},
        date={2006},
        ISSN={1312-5192},
     journal={J. Geom. Symmetry Phys.},
      volume={8},
       pages={1\ndash 16},
      review={\MR{2336854}},
}

\bib{AmannEscher}{book}{
      author={Amann, Herbert},
      author={Escher, Joachim},
       title={Analysis. {III}},
   publisher={Birkh\"{a}user Verlag, Basel},
        date={2009},
        ISBN={978-3-7643-7479-2; 3-7643-7479-2},
         url={https://doi.org/10.1007/978-3-7643-7480-8},
        note={Translated from the 2001 German original by Silvio Levy and
  Matthew Cargo},
      review={\MR{2500068}},
}

\bib{BauerKuwert}{article}{
      author={Bauer, Matthias},
      author={Kuwert, Ernst},
       title={Existence of minimizing {W}illmore surfaces of prescribed genus},
        date={2003},
        ISSN={1073-7928},
     journal={Int. Math. Res. Not.},
      number={10},
       pages={553\ndash 576},
         url={https://doi.org/10.1155/S1073792803208072},
      review={\MR{1941840}},
}

\bib{Birkhoff1917TAMS}{article}{
      author={Birkhoff, George~D.},
       title={Dynamical systems with two degrees of freedom},
        date={1917},
        ISSN={0002-9947,1088-6850},
     journal={Trans. Amer. Math. Soc.},
      volume={18},
      number={2},
       pages={199\ndash 300},
         url={https://doi.org/10.2307/1988861},
      review={\MR{1501070}},
}

\bib{Brendle21JAMS}{article}{
      author={Brendle, Simon},
       title={The isoperimetric inequality for a minimal submanifold in
  {E}uclidean space},
        date={2021},
        ISSN={0894-0347,1088-6834},
     journal={J. Amer. Math. Soc.},
      volume={34},
      number={2},
       pages={595\ndash 603},
         url={https://doi.org/10.1090/jams/969},
      review={\MR{4280868}},
}

\bib{ByrdFriedman}{book}{
      author={Byrd, Paul~F.},
      author={Friedman, Morris~D.},
       title={Handbook of elliptic integrals for engineers and scientists},
      series={Die Grundlehren der mathematischen Wissenschaften, Band 67},
   publisher={Springer-Verlag, New York-Heidelberg},
        date={1971},
        note={Second edition, revised},
      review={\MR{0277773}},
}

\bib{CalabiCao92}{article}{
      author={Calabi, Eugenio},
      author={Cao, Jian~Guo},
       title={Simple closed geodesics on convex surfaces},
        date={1992},
        ISSN={0022-040X},
     journal={J. Differential Geom.},
      volume={36},
      number={3},
       pages={517\ndash 549},
         url={http://projecteuclid.org/euclid.jdg/1214453180},
      review={\MR{1189495}},
}

\bib{Cheeger1970}{article}{
      author={Cheeger, Jeff},
       title={Finiteness theorems for {R}iemannian manifolds},
        date={1970},
        ISSN={0002-9327},
     journal={Amer. J. Math.},
      volume={92},
       pages={61\ndash 74},
         url={https://doi.org/10.2307/2373498},
      review={\MR{263092}},
}

\bib{Chen1971}{article}{
      author={Chen, Bang-yen},
       title={On a theorem of {F}enchel-{B}orsuk-{W}illmore-{C}hern-{L}ashof},
        date={1971},
        ISSN={0025-5831,1432-1807},
     journal={Math. Ann.},
      volume={194},
       pages={19\ndash 26},
         url={https://doi.org/10.1007/BF01351818},
      review={\MR{291994}},
}

\bib{Croke88}{article}{
      author={Croke, Christopher~B.},
       title={Area and the length of the shortest closed geodesic},
        date={1988},
        ISSN={0022-040X},
     journal={J. Differential Geom.},
      volume={27},
      number={1},
       pages={1\ndash 21},
         url={http://projecteuclid.org/euclid.jdg/1214441646},
      review={\MR{918453}},
}

\bib{CrokeKatzSurvey}{incollection}{
      author={Croke, Christopher~B.},
      author={Katz, Mikhail},
       title={Universal volume bounds in {R}iemannian manifolds},
        date={2003},
   booktitle={Surveys in differential geometry, {V}ol. {VIII} ({B}oston, {MA},
  2002)},
      series={Surv. Differ. Geom.},
      volume={8},
   publisher={Int. Press, Somerville, MA},
       pages={109\ndash 137},
         url={https://doi.org/10.4310/SDG.2003.v8.n1.a4},
      review={\MR{2039987}},
}

\bib{DeLellisMueller05}{article}{
      author={De~Lellis, Camillo},
      author={M\"uller, Stefan},
       title={Optimal rigidity estimates for nearly umbilical surfaces},
        date={2005},
        ISSN={0022-040X,1945-743X},
     journal={J. Differential Geom.},
      volume={69},
      number={1},
       pages={75\ndash 110},
         url={https://doi.org/10.4310/jdg/1121540340},
      review={\MR{2169583}},
}

\bib{DGR2017}{article}{
      author={Deckelnick, Klaus},
      author={Grunau, Hans-Christoph},
      author={R\"{o}ger, Matthias},
       title={Minimising a relaxed {W}illmore functional for graphs subject to
  boundary conditions},
        date={2017},
        ISSN={1463-9963},
     journal={Interfaces Free Bound.},
      volume={19},
      number={1},
       pages={109\ndash 140},
         url={https://doi.org/10.4171/IFB/378},
      review={\MR{3665920}},
}

\bib{Gage}{article}{
      author={Gage, Michael~E.},
       title={Curve shortening on surfaces},
        date={1990},
        ISSN={0012-9593},
     journal={Ann. Sci. \'{E}cole Norm. Sup. (4)},
      volume={23},
      number={2},
       pages={229\ndash 256},
         url={http://www.numdam.org/item?id=ASENS_1990_4_23_2_229_0},
      review={\MR{1046497}},
}

\bib{Grayson}{article}{
      author={Grayson, Matthew~A.},
       title={Shortening embedded curves},
        date={1989},
        ISSN={0003-486X},
     journal={Ann. of Math. (2)},
      volume={129},
      number={1},
       pages={71\ndash 111},
         url={https://doi.org/10.2307/1971486},
      review={\MR{979601}},
}

\bib{Gromov1983}{article}{
      author={Gromov, Mikhael},
       title={Filling {R}iemannian manifolds},
        date={1983},
        ISSN={0022-040X},
     journal={J. Differential Geom.},
      volume={18},
      number={1},
       pages={1\ndash 147},
         url={http://projecteuclid.org/euclid.jdg/1214509283},
      review={\MR{697984}},
}

\bib{Klingenberg61}{article}{
      author={Klingenberg, Wilhelm},
       title={\"{U}ber {R}iemannsche {M}annigfaltigkeiten mit positiver
  {K}r\"{u}mmung},
        date={1961},
        ISSN={0010-2571},
     journal={Comment. Math. Helv.},
      volume={35},
       pages={47\ndash 54},
         url={https://doi.org/10.1007/BF02567004},
      review={\MR{139120}},
}

\bib{Klingenber1995}{book}{
      author={Klingenberg, Wilhelm P.~A.},
       title={Riemannian geometry},
     edition={Second},
      series={De Gruyter Studies in Mathematics},
   publisher={Walter de Gruyter \& Co., Berlin},
        date={1995},
      volume={1},
        ISBN={3-11-014593-6},
         url={https://doi.org/10.1515/9783110905120},
      review={\MR{1330918}},
}

\bib{Kusner}{article}{
      author={Kusner, Rob},
       title={Comparison surfaces for the {W}illmore problem},
        date={1989},
        ISSN={0030-8730},
     journal={Pacific J. Math.},
      volume={138},
      number={2},
       pages={317\ndash 345},
         url={http://projecteuclid.org/euclid.pjm/1102650153},
      review={\MR{996204}},
}

\bib{LammSchaetzle14}{article}{
      author={Lamm, Tobias},
      author={Sch\"atzle, Reiner~Michael},
       title={Optimal rigidity estimates for nearly umbilical surfaces in
  arbitrary codimension},
        date={2014},
        ISSN={1016-443X,1420-8970},
     journal={Geom. Funct. Anal.},
      volume={24},
      number={6},
       pages={2029\ndash 2062},
         url={https://doi.org/10.1007/s00039-014-0303-6},
      review={\MR{3283934}},
}

\bib{Lee}{book}{
      author={Lee, John~M.},
       title={Introduction to {R}iemannian manifolds},
      series={Graduate Texts in Mathematics},
   publisher={Springer, Cham},
        date={2018},
      volume={176},
        ISBN={978-3-319-91754-2; 978-3-319-91755-9},
        note={Second edition of [ MR1468735]},
      review={\MR{3887684}},
}

\bib{LY}{article}{
      author={Li, Peter},
      author={Yau, Shing~Tung},
       title={A new conformal invariant and its applications to the {W}illmore
  conjecture and the first eigenvalue of compact surfaces},
        date={1982},
        ISSN={0020-9910},
     journal={Invent. Math.},
      volume={69},
      number={2},
       pages={269\ndash 291},
         url={https://doi.org/10.1007/BF01399507},
      review={\MR{674407}},
}

\bib{MarquesNeves14}{article}{
      author={Marques, Fernando~C.},
      author={Neves, Andr\'{e}},
       title={Min-max theory and the {W}illmore conjecture},
        date={2014},
        ISSN={0003-486X},
     journal={Ann. of Math. (2)},
      volume={179},
      number={2},
       pages={683\ndash 782},
         url={https://doi.org/10.4007/annals.2014.179.2.6},
      review={\MR{3152944}},
}

\bib{MenneScharrer18}{article}{
      author={Menne, Ulrich},
      author={Scharrer, Christian},
       title={An isoperimetric inequality for diffused surfaces},
        date={2018},
        ISSN={0386-5991,1881-5472},
     journal={Kodai Math. J.},
      volume={41},
      number={1},
       pages={70\ndash 85},
         url={https://doi.org/10.2996/kmj/1521424824},
      review={\MR{3777387}},
}

\bib{MichaelSimon}{article}{
      author={Michael, J.~H.},
      author={Simon, L.~M.},
       title={Sobolev and mean-value inequalities on generalized submanifolds
  of {$R\sp{n}$}},
        date={1973},
        ISSN={0010-3640,1097-0312},
     journal={Comm. Pure Appl. Math.},
      volume={26},
       pages={361\ndash 379},
         url={https://doi.org/10.1002/cpa.3160260305},
      review={\MR{344978}},
}

\bib{Petersen}{book}{
      author={Petersen, Peter},
       title={Riemannian geometry},
     edition={Third},
      series={Graduate Texts in Mathematics},
   publisher={Springer, Cham},
        date={2016},
      volume={171},
        ISBN={978-3-319-26652-7; 978-3-319-26654-1},
         url={https://doi.org/10.1007/978-3-319-26654-1},
      review={\MR{3469435}},
}

\bib{Pitts}{book}{
      author={Pitts, Jon~T.},
       title={Existence and regularity of minimal surfaces on {R}iemannian
  manifolds},
      series={Mathematical Notes},
   publisher={Princeton University Press, Princeton, N.J.; University of Tokyo
  Press, Tokyo},
        date={1981},
      volume={27},
        ISBN={0-691-08290-1},
      review={\MR{626027}},
}

\bib{Tristan}{article}{
      author={Rivi\`ere, Tristan},
       title={Lipschitz conformal immersions from degenerating {R}iemann
  surfaces with {$L^2$}-bounded second fundamental forms},
        date={2013},
        ISSN={1864-8258},
     journal={Adv. Calc. Var.},
      volume={6},
      number={1},
       pages={1\ndash 31},
         url={https://doi.org/10.1515/acv-2012-0108},
      review={\MR{3008339}},
}

\bib{MR2231630}{article}{
      author={Rotman, Regina},
       title={The length of a shortest closed geodesic and the area of a
  2-dimensional sphere},
        date={2006},
        ISSN={0002-9939},
     journal={Proc. Amer. Math. Soc.},
      volume={134},
      number={10},
       pages={3041\ndash 3047},
         url={https://doi.org/10.1090/S0002-9939-06-08297-9},
      review={\MR{2231630}},
}

\bib{Simon}{article}{
      author={Simon, Leon},
       title={Existence of surfaces minimizing the {W}illmore functional},
        date={1993},
        ISSN={1019-8385},
     journal={Comm. Anal. Geom.},
      volume={1},
      number={2},
       pages={281\ndash 326},
         url={https://doi.org/10.4310/CAG.1993.v1.n2.a4},
      review={\MR{1243525}},
}

\bib{Toro}{article}{
      author={Toro, Tatiana},
       title={Surfaces with generalized second fundamental form in {$L^2$} are
  {L}ipschitz manifolds},
        date={1994},
        ISSN={0022-040X},
     journal={J. Differential Geom.},
      volume={39},
      number={1},
       pages={65\ndash 101},
         url={http://projecteuclid.org/euclid.jdg/1214454677},
      review={\MR{1258915}},
}

\bib{Volkmann}{article}{
      author={Volkmann, Alexander},
       title={A monotonicity formula for free boundary surfaces with respect to
  the unit ball},
        date={2016},
        ISSN={1019-8385},
     journal={Comm. Anal. Geom.},
      volume={24},
      number={1},
       pages={195\ndash 221},
         url={https://doi.org/10.4310/CAG.2016.v24.n1.a7},
      review={\MR{3514558}},
}

\bib{Willmore_Riemannian}{book}{
      author={Willmore, Thomas~J.},
       title={Riemannian geometry},
      series={Oxford Science Publications},
   publisher={The Clarendon Press, Oxford University Press, New York},
        date={1993},
        ISBN={0-19-853253-9},
      review={\MR{1261641}},
}

\end{biblist}
\end{bibdiv}
\bibliographystyle{abbrev}
\end{document}